\documentclass{amsart}

\usepackage{indentfirst}
\usepackage{hyperref}
\usepackage{color}
\usepackage{cite}
\usepackage{graphicx,float}
\usepackage[sectionbib]{chapterbib}
\usepackage{enumerate}
\usepackage{ulem}

\usepackage{amsmath}
\usepackage{amsfonts}
\usepackage{mathdots}
\usepackage{mathrsfs}
\usepackage{amssymb}
\usepackage{amsthm}
\usepackage{extarrows}
\usepackage{cleveref}

\usepackage{amscd}
\usepackage[all]{xy}
\usepackage{tikz-cd}

\hypersetup{
    colorlinks=true,
    linkcolor=blue,   % \Cref, \ref 等
    citecolor=blue,   % \cite
    urlcolor=blue     % 可选：网址
}

\theoremstyle{plain}
\newtheorem{theorem}{Theorem}[section]
\newtheorem*{theorem*}{Theorem}
\newtheorem{corollary}[theorem]{Corollary}
\newtheorem{lemma}[theorem]{Lemma}
\newtheorem{proposition}[theorem]{Proposition}

\theoremstyle{definition}

\newtheorem{example}[theorem]{Example}

\theoremstyle{remark}
\newtheorem{remark}[theorem]{Remark}

\title{On the equivariant $KU_G$-local sphere for finite abelian groups}
\author{Yingxin Li}
\address{School of Mathematical Sciences, Nankai University, Tianjin, China}
\email{yingxinli@mail.bnu.edu.cn}

\begin{document}

\begin{abstract}
We study $KU_G$-localization for finite groups. For a finite nilpotent group $G$, we identify the $KU_G/p$-local sphere as the fiber of an Adams operation, and reduce the computation of its homotopy Mackey functors to the corresponding computation for the Sylow $p$-subgroup of $G$. At the prime $2$, we compute $\underline{\pi}_* L_{KU_G/2}S_G$ for abelian $2$-groups, resolve the resulting extension problems, and determine the degree-zero Hurewicz image. As a consequence, we completely determine the $\mathbb Z$-graded homotopy Mackey functors of $L_{KU_G}S_G$ for all finite abelian groups. Finally, for an arbitrary finite group $G$, we prove that the $KU_G/p$-localization of $G$-spectra splits as a wedge of height-one equivariant Morava $K$-localizations, indexed by conjugacy classes of cyclic subgroups of order prime to $p$.
\end{abstract}

\maketitle
\tableofcontents

%%------------------------------------------------------------------------------------------------------------------------------------------------------------------------------------

\section{Introduction}
Non-equivariantly, chromatic homotopy theory provides a systematic framework for studying periodic phenomena in stable homotopy theory, organized by chromatic height. Equivariantly, chromatic homotopy theory admits a natural equivariant refinement. This viewpoint suggests that one can understand equivariant stable homotopy groups by studying the chromatic filtration in the category of equivariant spectra. The foundations of equivariant chromatic homotopy theory were developed in \cite{cole2000equivariant,strickland2011multicurves,hanke2018equivariant,barthel2020balmer,hu2021equivariant,hausmann2022global,hausmann2023invariant,wisdom2024properties}; see also the survey by Behrens and Carlisle\cite{behrens2024periodic}.

At chromatic height $1$, non-equivariant $v_1$-periodicity is reflected in the Bott periodicity of topological $KO$-theory and can be studied through the $KU$-local sphere $L_{KU}S$. For a finite group $G$, both $KO$ and $L_{KU}S$ admit natural $G$-equivariant refinements, and several computations of the homotopy groups of the $KU_G$-local sphere $L_{KU_G}S_G$ have been established in recent years. Balderrama\cite{balderrama2021c_2} computed the $RO(C_2)$-graded homotopy $C_2$-Green functor of $L_{KU_{C_2}/2}S_{C_2}$. For a finite $p$-group $G$ with $p$ odd, Carawan et al.\cite{carawan2023homotopy} computed the $\mathbb{Z}$-graded homotopy Mackey functors of $L_{KU_G}S_G$, and Balderrama\cite{balderrama2024total} investigated the norm maps in $\underline{\pi}_0L_{KU_G}S_G$. 
Related height-one phenomena have also been studied through the equivariant $J$-homomorphism\cite{french2009equivariant,balderrama2023equivalences} and periodic self-maps for $G=C_2$\cite{quigley2021real, bhattacharya2022} and $G=C_{p^n}$\cite{balderrama2024cpn}. These results leave open, in particular, the computation for finite nilpotent groups of mixed prime order, as well as the extension problems that arise at the prime $2$. The present paper addresses both of these issues.

It is worth noting that, for a finite group $G$, knowledge of the homotopy groups of the $G$-equivariant sphere spectrum remains very limited. For $G=C_2$, the $RO(C_2)$-graded ring $\pi_{\star}^{C_2}S_{\mathbb{Q}}$ was computed by Belmont-Xu-Zhang\cite{belmont2024reduced}, and $\pi_\star^{C_2} S_{C_2}$ has been computed in a range of degrees in \cite{araki1982equivariant, iriye1982equivariant, dugger2017Z, belmont2022r, belmont2021C2, ma2022borel, guillou2024c_2}. For $G=C_3$, Hou-Zhang\cite{hou2025c_3} carried out partial computations of $\pi_\star^{C_3} S_{C_3}$. Beyond these cases, very few computations are known. This makes chromatic localizations of the equivariant sphere, such as $L_{KU_G}S_G$, natural and more tractable to study.

In this paper, we study $L_{KU_G}S_G$ for finite groups. For finite nilpotent groups $G$, we reduce the computation of the homotopy Mackey functors of $L_{KU_G/p}S_G$ to the corresponding computation for the Sylow $p$-subgroup of $G$. Combining this reduction with the known odd-primary computations and our analysis at the prime $2$, we compute the $\mathbb{Z}$-graded homotopy Mackey functors of $L_{KU_G}S_G$ for finite abelian groups $G$. For an arbitrary finite group $G$, we further identify the Bousfield class of $KU_G/p$ in terms of height-one equivariant Morava $K$-theories and prove that, for every $G$-spectrum $X$, the $KU_G/p$-localization of $X$ splits as a wedge of height-one equivariant Morava $K$-localizations.

\subsection{Statement of main results}
When $G$ is a $p$-group for an odd prime $p$, Carawan-Field-Guillou-Mehrle-Stapleton\cite{carawan2023homotopy} show that, for a generator $g$ of $(\mathbb{Z}_p^\wedge)^\times$, there is a fiber sequence
\[
L_{KU_G/p}S_G \longrightarrow {(KU_{G})}_{p}^{\wedge}\overset{\psi^g-1}{\longrightarrow} {(KU_{G})}_p^{\wedge},
\]
and they use this fiber sequence to compute the homotopy groups of $L_{KU_G/p}S_G$. However, if $G$ is not a $p$-group, this sequence is no longer a fiber sequence since $KU_G/p$ has nontrivial geometric fixed points at nontrivial cyclic subgroups of order prime to $p$. For a finite nilpotent group $G$, the product decomposition into Sylow subgroups allows us to isolate the $p$-primary part, which leads to the fiber sequence in the following proposition.

\begin{proposition}[\Cref{thm:KUG local spectrum first case}]\label{IntroThm:KU_G/p local sphere I}
    Let $G$ be a finite nilpotent group, and let $Cyc$ be the family of all cyclic subgroups of $G$. For any prime $p$, let $N_p$ be the Sylow $p$-subgroup of $G$, and let $g$ be a topological generator of $\mathbb{Z}_p^{\times}/\{\pm 1\}$. Then for any finite $G$-spectrum $X$, there is a fiber sequence 
    \[
    L_{KU_G/p}X \longrightarrow (ECyc_+\wedge\ \operatorname{Inf}_{N_p}^G KO_{N_p}\wedge X)_p^{\wedge}\overset{\psi^g-1}{\longrightarrow} (ECyc_+\wedge\ \operatorname{Inf}_{N_p}^G KO_{N_p}\wedge X)_p^{\wedge} .
    \]
\end{proposition}

As a corollary, we prove that for any finite nilpotent group $G$,
    \[    
    L_{KU_G/p}S_G\simeq (ECyc_+\wedge \operatorname{Inf}_{N_p}^G L_{KU_{N_p}/p}S_{N_p})_p^{\wedge}. 
    \] 
For any subgroup $H\subset G$, let $\underline{A}_H$ denote the $H$-Mackey functor given by the Burnside ring, and let $\underline{J}_H$ denote the subfunctor of $\underline{A}_H$ given by the Brauer relations (see \Cref{Sec:fixed point of KU/p local sphere}). Define
\[ \underline{A/J}_H:=\underline{A}_H/\underline{J}_H .\]
When $H=G$, we omit the subscript and write $\underline{A/J}$ for $\underline{A/J}_G$. In order to compute the homotopy groups of $L_{KU_G/p}S_G$, we study the fixed points of $G$-spectra of $(ECyc_+\wedge \operatorname{Inf}_{N_p}^G E)_p^\wedge$ for an $N_p$-spectrum $E$, and prove the following proposition.

\begin{proposition}[\Cref{lem:Z graded homotopy Mackey functor of ECyc inf}]\label{IntroProp: Z graded homotopy Mackey functor of ECyc inf}
    Let $G$ be a finite nilpotent group with Sylow $p$-subgroup $N_p$, and let $N$ be the subgroup of $G$ such that $G\cong N_p\times N$. Then for any $N_p$-spectrum $E$ such that $E$ is $S/p$-equivalent to $ECyc_+ \wedge E$, we have an isomorphism between $G$-Mackey functors
    \[ \underline{\pi}_*(ECyc_+\wedge \operatorname{Inf}_{N_p}^G E)_p^\wedge\cong (\underline{\pi}_* E \otimes \underline{A/J}_N)_p^\wedge. \] 
    Here the symbol $\otimes$ denotes the external tensor product of Mackey functors defined in \Cref{lem:tensor product}.
\end{proposition}

    When $E=L_{KU_{N_p}/p}S_{N_p}$, we have
    \[\underline{\pi}_*(L_{KU_G/p}S_G)\cong \underline{\pi}_*L_{KU_{N_p}/p}S_{N_p}\otimes_{\mathbb{Z}_p} (\underline{A/J}_N)_p^\wedge.\]
    
    \begin{remark}
    The preceding result reduces the computation for a finite nilpotent group to the computation for its Sylow $p$-subgroup. For odd primes, this computation was carried out in \cite{carawan2023homotopy}. At the prime $2$, however, the long exact sequence associated to the Adams-operation fiber determines the homotopy Mackey functors only up to extension in certain degrees. Thus, even after computing the kernel and cokernel of $\psi^g-1$, additional information is required to recover the Mackey functor structure. Resolving these extension problems is the main difficulty in our $2$-primary computation.
    \end{remark}

    It therefore remains to determine $\underline{\pi}_*L_{KU_{N_2}/2}S_{N_2}$. When $N_2$ is abelian, we can compute it via the short exact sequence
    \[0\longrightarrow \underline{\mathrm{coker}}_2\{k+1\}\longrightarrow \underline{\pi}_k L_{KU_{N_2}/2}S_{N_2} \longrightarrow \underline{\ker}_2\{k\}\longrightarrow 0,\] 
    where   
    \[\underline{\ker}_2\{k\}:=\ker(\underline{\pi}_k(KO_{N_2})_2^{\wedge}\overset{\psi^g-1}{\longrightarrow}\underline{\pi}_k (KO_{N_2})_2^{\wedge}),\]
    \[ \underline{\mathrm{coker}}_2\{k\}:=\mathrm{coker}(\underline{\pi}_k (KO_{N_2})_2^{\wedge}\overset{\psi^g-1}{\longrightarrow}\underline{\pi}_k (KO_{N_2})_2^{\wedge}).\]
    Extension problems occur in degrees $0$ and $8d+1$, $d\in \mathbb Z$. Although these extensions split after evaluation at each orbit as extensions of abelian groups, the resulting splittings need not be compatible with the Mackey functor structure. Let $\underline{RO(-;\mathbb{R})}$ be the Mackey functor whose value at $N_2/H$ is the free abelian group generated by the irreducible real $H$-representations of real type. %whose endomorphism rings are isomorphic to $\mathbb{R}$.  
    The degree-zero extension is
    \[
    0\longrightarrow
    \underline{RO(-;\mathbb{R})}_{N_2}/2
    \longrightarrow
    \underline{\pi}_0L_{KU_{N_2}/2}S_{N_2}
    \longrightarrow
    (\underline{A/J}_{N_2})_2^\wedge
    \longrightarrow 0.
    \]
    The key point is that although an element of $\underline{J}_{N_2}$ vanishes in $\underline{A/J}_{N_2}$, the Hurewicz map
    \[
    \underline{A}_{N_2}
    \cong \underline{\pi}_0S_{N_2}
    \longrightarrow
    \underline{\pi}_0L_{KU_{N_2}/2}S_{N_2}
    \]
    can still detect it as a $2$-torsion class corresponding to an element in \(\underline{RO(-;\mathbb{R})}_{N_2}/2\). The case of $N_2=C_2\times C_2$ studied by Szymik \cite[Example 5.1]{szymik2013chromatic} gives a basic example. This phenomenon is captured by a map of Mackey functors
    \[
    \theta_{N_2}\colon \underline{J}_{N_2} \longrightarrow \underline{\mathrm{coker}}_2\{1\}\cong \underline{RO(-;\mathbb{R})}_{N_2}/2
    \]
    which is induced by the Hurewicz map. We compute $\theta_{N_2}$ for all finite abelian $2$-groups $N_2$ via the Brauer relations, which resolves the degree-zero extension problem. Equivalently, this determines the degree-zero Hurewicz image of $L_{KU_{N_2}/2}S_{N_2}$.

    In degree $8d+1$, the remaining extension data are similarly encoded by a morphism
    \[
    \theta_{8d+1}\colon
    \underline I_{8d+1}\longrightarrow
    \underline{\operatorname{coker}}_2\{8d+2\},
    \]
    where $\underline I_{8d+1}$ is the kernel of the mod-$2$ linearization map $\underline{A}_{N_2}/2\to \underline{RO(-;\mathbb R)}_{N_2}/2$. We determine $\theta_{8d+1}$ by combining the mod-$2$ Brauer relations with an explicit computation of $\underline{\pi}_1L_{KU_{C_4}/2}S_{C_4}$, thereby resolving the extension problem in degrees $8d+1$.

    Our computation of $\underline{\pi}_* L_{KU_G/p}S_G$ is summarized in \Cref{lem:extension of 2 group KU local sphere,thm: Z graded homotopy of KU_G/p local sphere}. Then we compute $\underline{\pi}_\ast L_{KU_G}S_G$ for finite abelian groups $G$ via the arithmetic fracture square. 

    \begin{theorem}[\Cref{thm:integer graded homotopy Mackey functor of ku local sphere}]\label{IntroThm:integer graded homotopy Mackey functor of ku local sphere}
	Let $G$ be a finite abelian group. For each prime $p$, let $N_p\subset G$ be its Sylow $p$-subgroup, and identify $G/N_p \cong \prod_{q\neq p}N_q$. Then for every $k\in \mathbb Z$, there is an isomorphism of $G$-Mackey functors
    \[
	\underline{\pi}_kL_{KU_G}S_G\cong \begin{cases}
		\underline{A/J}_{G/N_2} \otimes  \frac{\underline{A}_{N_2} \oplus \underline{RO(-;\mathbb{R})}_{N_2}\{\eta\}/2}{\{j-\theta_{N_2}(j):j\in \underline{J}_{N_2}\}} & \quad k=0\\
		0 & \quad k=-1\\
		\mathbb{Q}/\mathbb{Z}\otimes (\prod_p \underline{\operatorname{coker}}_p\{0\}\otimes \underline{A/J}_{G/N_p}) & \quad k=-2\\
		\prod_p \underline{\pi}_kL_{KU_G/p}S_G & \quad \text{otherwise}
	\end{cases}.
	\]
    The computation of $\theta_{N_2}$ is carried out in \Cref{prop:extension of 2 group KU local sphere-degree-zero}.
    The functors $\underline{\operatorname{coker}}_p\{0\}$ are listed in \Cref{lem:psi-1,lem:psi-1 for odd p}.
    \end{theorem}

    Finally, fix a prime $p$, we study $KU_G/p$ from the viewpoint of equivariant chromatic homotopy theory for an arbitrary finite group $G$. Extending Strickland's construction of equivariant Morava $K$-theory for finite abelian groups \cite{strickland2011multicurves}, we define $K(H,n)$ for each subgroup $H\subset G$ and integer $n\geq 0$, and give an explicit construction of the $K(H,n)$-localization of $G$-spectra. The comparison of the Bousfield classes of $KU_G/p$ and those of the height-one $K(H,1)$, together with a construction of $K(H,1)$-localization of $G$-spectra, leads to the following splitting of $KU_G/p$-localization. In contrast to the preceding computations, this splitting holds for arbitrary finite groups.
    \begin{theorem}[\Cref{thm:KUG local spectrum second case}]\label{IntroThm:KUG local spectrum second case}
        Let $G$ be a finite group, and let $Cyc$ be the family of cyclic subgroups of $G$. For any prime $p$ and any $G$-spectrum $X$, there is an equivalence of $G$-spectra
    \[
    L_{KU_G/p}X \simeq L_{\bigvee_{(H)\in Cyc/G, p\nmid |H|}K(H,1)}X \simeq \bigvee_{(H)\in Cyc/G,\, p\nmid |H|} L_{K(H,1)}X.
    \]
    Here $H$ ranges over the conjugacy classes of cyclic subgroups of $G$ such that $p\nmid |H|$.
    \end{theorem}
    As an application, when $G$ is a finite nilpotent group with Sylow $p$-subgroup $N_p$, we show that the computation of $\underline{\pi}_V L_{KU_G/p}S_G$ for $V\in RO(G)$ reduces to that of $\underline{\pi}_{V^H}L_{KU_{N_p}/p}S_{N_p}$ for any cyclic subgroup $H$ such that $p\nmid |H|$; see \Cref{cor:ROG graded KU_G/p local sphere}. More broadly, we expect that the construction of $K(H,n)$-localization will be useful for investigating the localization with respect to other global spectra.

    \begin{remark}
        The reduction of the $KU_G/p$-local sphere from a finite nilpotent group $G$ to its Sylow $p$-subgroup $N_p$ is compatible with the $p$-primary behavior of the equivariant image of $J$ studied by French \cite{french2009equivariant}, which suggests a close relationship between the equivariant image of $J$ and the $KU_G/p$-local sphere. Moreover,  our computation determines the degree-zero Hurewicz image of $L_{KU_G}S_G$ for finite abelian group $G$, which shows that the Hurewicz image in the $KU_G/2$-local sphere can detect torsion arising from Brauer relations in the Burnside ring that is invisible in the corresponding rational representation ring. It would be interesting to understand how this phenomenon is reflected in the equivariant image of $J$. 
    \end{remark}

    \subsection{Outline}
    In \Cref{Sec:Preliminaries}, we recall background on equivariant stable homotopy theory and Bousfield localization.
    In \Cref{Sec: KU/p local sphere}, we study the Bousfield classes of $KU_G/p$ and $KO_G/p$ for a finite nilpotent $G$, and prove \Cref{IntroThm:KU_G/p local sphere I}. 
    In \Cref{Sec:fixed point of KU/p local sphere}, for a finite nilpotent group $G$ and its Sylow subgroup $N_p$, we study the fixed points of $ECyc_+\wedge \operatorname{Inf}_{N_p}^G E$ for an $N_p$-spectrum $E$, and prove \Cref{IntroProp: Z graded homotopy Mackey functor of ECyc inf}. 
    In \Cref{Sec:computation for 2-group}, we compute $\underline{\pi}_\ast L_{KU_{N_2}/2}S_{N_2}$ for a finite abelian $2$-group $N_2$, with particular attention to the extension problems in degrees $0$ and $8d+1$. 
    In \Cref{Sec:KU_G local sphere}, by combining \Cref{IntroProp: Z graded homotopy Mackey functor of ECyc inf} with the computation in \Cref{Sec:computation for 2-group}, we obtain $\underline{\pi}_\ast L_{KU_G/2}S_G$ for finite abelian groups $G$. Then we prove \Cref{IntroThm:integer graded homotopy Mackey functor of ku local sphere}. 
    In \Cref{Sec:equivariant chromatic}, we introduce and study the spectra $K(H,n)$, compare their Bousfield classes with that of $KU_G/p$, prove the splitting in \Cref{IntroThm:KUG local spectrum second case} for arbitrary finite groups, and apply it to the $RO(G)$-graded homotopy Mackey functors of the $KU_G/p$-local sphere for finite nilpotent groups. 
    In \Cref{sec:C4}, we compute the $C_4$-Mackey functor $\underline{\pi}_1 L_{KU_{C_4}/2}S_{C_4}$, which is used to resolve the extension problem in degree $8d+1$.

    \subsection*{Acknowledgements}

    This paper is based on part of my PhD thesis. I would like to thank my advisor, Xu-an Zhao, for his guidance throughout this project and for carefully proofreading the manuscript. I am grateful to Markus Hausmann for inspiring my initial interest in equivariant chromatic homotopy theory. I would also like to thank William Balderrama for his valuable suggestions regarding the extension problems for Mackey functors. I also thank Tao Gong, Guchuan Li, Xuecai Ma, Zezhou Zhang, Yifei Zhu and Foling Zou for helpful discussions and comments.

%%-------------------------------------------------------------------------------------------------------------------------------------------------------------------------------------
\section{Preliminaries}\label{Sec:Preliminaries}
In this section, we briefly review some results that will be used throughout the paper. We refer the reader to \cite{may1996equivariant,schwede2019lectures} for background on equivariant homotopy theory, and to \cite{bousfield1979localization,hill2019equivariant,carrick2022smashing} for background on Bousfield localization.

\subsection{Equivariant stable homotopy theory}
 For a finite group $G$, let $\mathcal{S}^G$ be the category of $G$-spaces, and let $\mathrm{Sp}^G$ be the category of genuine $G$-spectra. For a $G$-space $X$, we write $\Sigma_G^\infty X$ for the suspension $G$-spectrum of $X$. Let $S_G$ be the $G$-equivariant sphere spectrum. For any $G$-spectrum $X$, let $\underline{\pi}_* X$ be the homotopy Mackey functors of $X$, where
\[
\underline{\pi}_n X (G/H)=\pi_n^H(X) \cong [G/H_+\wedge S^n, X]_G
\]
for any $G$-orbit $G/H$.

Let $\alpha\colon H\to G$ be a homomorphism of finite groups. Any $G$-spectrum can be regarded as an $H$-spectrum via $\alpha$, yielding a symmetric monoidal functor $\alpha^*\colon \mathrm{Sp}^G \to \mathrm{Sp}^H$.
In particular:
\begin{itemize}
\item[(1)] If $\alpha:H\hookrightarrow G$ is inclusion, we denote $\alpha^*$ by $\operatorname{Res}_H^G$. The restriction functor $\operatorname{Res}_H^G$ admits both a left adjoint $\operatorname{Ind}_H^G$ and a right adjoint $\operatorname{CoInd}_H^G$.
\item[(2)] If $N\trianglelefteq G$ is a normal subgroup and $\alpha\colon G\to G/N$ is the quotient map, we denote $\alpha^*$ by $\operatorname{Inf}_{G/N}^G$. The inflation functor $\operatorname{Inf}_{G/N}^G$ is left adjoint to the $N$-fixed point functor
$(-)^N\colon \mathrm{Sp}^G \to \mathrm{Sp}^{G/N}$. 
For $X\in \mathrm{Sp}^{G/N}$ and $Y\in \mathrm{Sp}^G$, there is an equivalence
\[(\operatorname{Inf}_{G/N}^G(X)\wedge Y)^N\simeq X\otimes Y^N.\]
\end{itemize}

There is also a commonly used functor called the geometric fixed point functor. Let $\mathcal{F}$ be a family of subgroups of $G$, there is an unbased $G$-space $E\mathcal{F}$ such that 
    \[
    (E\mathcal{F})^H=\begin{cases}
        pt &  \text{if }\ H\in\mathcal{F}\\
        \emptyset &  \text{if }\ H\not\in\mathcal{F},
    \end{cases}
    \]
and let $\widetilde{E\mathcal{F}}$ be the cofiber of $E\mathcal{F}_+\to S^0_G$. For any subgroup $H\subset G$, let $\mathcal{F}_{H\not\subset}$ denote the family of subgroups of $G$ that contain no $G$-conjugates of $H$. The $H$-geometric fixed point of a $G$-spectrum $X$ is given by 
\[
\Phi^H(X):=(X\wedge \widetilde{E\mathcal{F}_{H\not\subset}})^H.
\]
If $H$ is a normal subgroup of $G$, $\Phi^H(X)$ has a residual action of $G/H$, one can regard the $H$-geometric fixed point as a functor $\Phi^H:\mathrm{Sp}^G\to \mathrm{Sp}^{G/H}$. Moreover, there is an equivalence of functors $\Phi^H\circ \operatorname{Inf}_{G/H}^G\simeq \operatorname{Id}_{\mathrm{Sp}^{G/H}}$.

For any finite group $G$, topological $K$-theory admits a natural $G$-equivariant refinement. Let $KU_G$ (resp., $KO_G$) be the $G$-equivariant complex (resp., real) $K$-theory. Let $RU(G)$ (resp., $RO(G)$) be the complex (resp., real) representation ring of $G$, and let $\underline{RU}$ be the $G$-Green functor such that $\underline{RU}(G/H):=RU(H)$. By \cite[Proposition 2.2]{segal1968equivariant}, we have
\[
\underline{\pi}_* KU_G\cong \underline{RU}[\beta^{\pm}].
\]
We denote by $RO(G;\mathbb{R})$ (resp., $RO(G;\mathbb{C})$, $RO(G;\mathbb{H})$) the free abelian group generated by irreducible real $G$-representations whose endomorphism ring is isomorphic to $\mathbb{R}$ (resp., $\mathbb{C}$, $\mathbb{H}$). Then 
\[
\pi_* KO_G \cong \bigoplus_{\mathbb{F}=\mathbb{R},\mathbb{C},\mathbb{H}}\pi_*K\mathbb{F}\otimes RO(G;\mathbb{F}),
\]
where
\[
K\mathbb{F}=\begin{cases}
    KO & \mathbb{F}=\mathbb{R},\\
    KU & \mathbb{F}=\mathbb{C},\\
    KSp & \mathbb{F}=\mathbb{H}.
\end{cases}
\]
The restriction and transfer homomorphisms in $\underline{\pi}_* KO_G$ can be computed via the complexification map $KO_G\to KU_G$; see \cite[Section 9.1]{mathew2017nilpotence} for details.

    \begin{proposition}\cite[Proposition 7.7.7]{dieck2006transformation}\label{prop:geometric fixed KU}
    For any finite group $G$, 
    $$\Phi^{G} KU_{G}\;\simeq\;
        \begin{cases}
            KU\otimes \mathbb{Z}\left[\frac{1}{n},\zeta_{n}\right], & \text{if } G \cong C_{n}\\[6pt]
            \ast, & \text{otherwise}
        \end{cases},$$
        where $\zeta_n$ is a primitive $n$-th root of unity.
    \end{proposition}

%%-----------------------------------------------------------------------------------------
\subsection{Bousfield localization} 
Let $G$ be a finite group. For $E,X\in \mathrm{Sp}^G$, the Bousfield localization of $X$ with respect to $E$ is an $E$-equivalence $f:X\to L_EX$ such that $L_EX$ is $E$-local.  For $E,F\in \mathrm{Sp}^G$, we say that $E$ and $F$ are Bousfield equivalent, and write $\langle E\rangle=\langle F\rangle$, if
\[
\forall X\in \mathrm{Sp}^G, \quad X\wedge E\simeq \ast \ \Longleftrightarrow \  X\wedge F\simeq \ast.
\]
If $\langle E\rangle=\langle F\rangle$, then $L_EX\simeq L_FX$ for every $X\in \mathrm{Sp}^G$. For any prime $p$, a $G$-spectrum $X$ is called $p$-local if $X\simeq L_{S_G\otimes\mathbb{Z}{(p)}}X$, and $p$-complete if $X\simeq L_{S_G/p}X$.

    By \cite[Proposition 3.2]{carrick2022smashing}, for any subgroup $H\subset G$ and any $X\in \mathrm{Sp}^G$, there is an equivalence
    \[
        \operatorname{Res}_H^G L_EX\simeq L_{\operatorname{Res}_H^G E} \operatorname{Res}_H^G X. 
    \]
    When $p\nmid |H|$, by the splitting of the category of $p$-local $H$-spectra, it follows from \cite[Proposition 8.5]{bonventre2022ku_g} that, for any $X,E\in \mathrm{Sp}^H_{(p)}$,
    \[
        \Phi^H L_E X\simeq L_{\Phi^H E} \Phi^H X
    \]
    as non-equivariant spectra.
    As a result, for any $H\subset G$ such that $p\nmid |H|$ and for any $X, E\in \mathrm{Sp}^G_{(p)}$, we have 
    \[\Phi^H L_EX\simeq \Phi^H \operatorname{Res}_H^GL_EX\simeq  \Phi^H L_{\operatorname{Res}_H^G E}\operatorname{Res}_H^G X\simeq L_{\Phi^H E}\Phi^H X \in \mathrm{Sp}_{(p)}.\]
In particular, we can show that $\Phi^H: \mathrm{Sp}^G\to \mathrm{Sp}$ preserves $p$-completion if $p\nmid |H|$.
\begin{proposition}\label{prop:q local geometric fixed points}
    For any $X\in \mathrm{Sp}^G$ and any subgroup $H\subset G$ with $p\nmid |H|$, there is an equivalence $\Phi^H (X_p^\wedge)\simeq (\Phi^H X)_p^\wedge$ of non-equivariant spectra. 
\end{proposition}
\begin{proof}
    Since every $S\mathbb{Z}_{(p)}$-acyclic $G$-spectrum is $S/p$-acyclic, we have $L_{S/p}L_{S\mathbb{Z}_{(p)}}\simeq L_{S/p}$. If $p\nmid |H|$, we  have 
    \[
    \begin{aligned}
        \Phi^H L_{S/p}X & \simeq \Phi^H L_{S/p} L_{S\mathbb{Z}_{(p)}}X_{(p)} \simeq L_{\Phi^H S/p} \Phi^H(L_{S\mathbb{Z}_{(p)}}X)\\
        & \simeq L_{S/p} L_{S\mathbb{Z}_{(p)}}\Phi^H X\simeq L_{S/p} \Phi^H X.
    \end{aligned} 
    \]
    Here $\Phi^H L_{S\mathbb{Z}_{(p)}}X\simeq L_{S\mathbb{Z}_{(p)}}\Phi^H X$ since $L_{S\mathbb{Z}_{(p)}}$ is smashing.
\end{proof}

\begin{proposition}\cite{bauer2011bousfield}\label{prop:arithmetic square}
    Let $E,F,X\in \mathrm{Sp}^G$. If $E\wedge L_F X\simeq \ast$, the following diagram is a pullback.
    \[
        \begin{tikzcd}
        L_{E\vee F}X \arrow[r] \arrow[d] & L_E X \arrow[d] \\
        L_F X \arrow[r] & L_F L_E X.
        \end{tikzcd}
    \]
\end{proposition}
In particular, let $E=\bigvee_p S_G/p$, $F=S_G\otimes\mathbb{Q}$, we have the arithmetic fracture square
    \[
        \begin{tikzcd}
        X \arrow[r] \arrow[d] & \prod_{p} X_{p}^{\wedge} \arrow[d] \\
        X_{\mathbb{Q}} \arrow[r] & (\prod_{p} X_{p}^{\wedge})_{\mathbb{Q}} .
        \end{tikzcd}
    \]
For any $E,X\in \mathrm{Sp}^G$, we have $L_{E/p}X\simeq (L_E X)_p^\wedge$ and $L_{E\otimes \mathbb{Q}}X\simeq (L_E X)_{\mathbb{Q}}$.
Therefore, $L_E X$ can be recovered from $L_{E/p}X$ and $L_{E\otimes \mathbb{Q}}X$ via the arithmetic fracture square.
  
\begin{proposition}\label{prop:mod p local}\cite[Lemma 6.2]{bonventre2022ku_g}
    Let $R$ be a $G$-equivariant ring spectrum. Then for any finite $G$-spectrum $X$, $(R\wedge X)_p^{\wedge}$ is $R/p$-local.
\end{proposition}
\begin{proof}
    The case $X=S_G$ is treated in \cite[Lemma 6.2]{bonventre2022ku_g}, and the same proof applies to general finite $G$-spectra. For any $R/p$-acyclic $G$-spectrum $M$, $M \wedge R$ is $S/p$-acyclic, hence
    \[ [M,(R\wedge X)_p^{\wedge}]_G\cong [M\wedge R,(R\wedge X)_p^{\wedge}]_{G}^{R-\mathrm{Mod}}\subset [M\wedge R,(R\wedge X)_p^{\wedge}]_G=0.\]
\end{proof}

Finally, we give two examples of Bousfield localization required for this paper.

\begin{example}
Let $G$ be a finite group, and let $H\subset G$ be a subgroup. Let $\mathcal{F}_H$ denote the smallest family of subgroups of $G$ that contains $H$ as an element. Then 
\[L_{G/H_+}X \simeq F((E\mathcal{F}_H)_+, X)\simeq F((EG/H)_+, X).\]
Indeed, $F((E\mathcal{F}_H)_+, X) \to F((E\mathcal{F}_H)_+, Y)$ is a $G$-equivalence if and only if $X \to Y$ is an $H$-equivalence. Thus, a $G$-spectrum $X$ satisfies $X \simeq F((E\mathcal{F}_H)_+, X)$ if and only if $X$ is $G/H_+$-local, and $F((E\mathcal{F}_H)_+, X) \simeq L_{G/H_+} X$. 
%A $G$-spectrum $X$ is said to be $H$-cofree if $X \simeq F((E\mathcal{F}_H)_+, X)$.
\end{example}

\begin{example}
    Let $\mathcal{F}$ be a family of subgroups of $G$, and let $X \in \mathrm{Sp}^G$. We have 
    \[
    L_{\widetilde{E\mathcal{F}}}X\simeq X \wedge \widetilde{E\mathcal{F}}.
    \]
    Indeed, define
$X[\mathcal{F}^{-1}] := X \wedge \widetilde{E\mathcal{F}}$.
Then the natural map $X\to X[\mathcal{F}^{-1}]$ is a $\widetilde{E\mathcal{F}}$-equivalence, and $X[\mathcal{F}^{-1}]$ is $\widetilde{E\mathcal{F}}$-local, since $\widetilde{E\mathcal{F}}\wedge \widetilde{E\mathcal{F}}\simeq \widetilde{E\mathcal{F}}$. 
\end{example}

\begin{proposition}\label{prop:pushout type}
    Let $G$ be a finite group and let $H\triangleleft G$ be a normal subgroup. Let $E$ be a $G/H$-equivariant ring spectrum, and define $E_G := \bigl(\operatorname{Inf}^{G}_{G/H} E\bigr)\bigl[\mathcal{F}_{H\not\subset}^{-1}\bigr]$. Then for every $G$-spectrum $X$, there are natural isomorphisms
    \[
    E_G^{*}(X)\cong E^{*}(\Phi^{H}X), \qquad (E_G)_{*}(X)\cong E_{*}(\Phi^{H}X).
    \]
\end{proposition}
\begin{proof}
    The case where $G$ is a finite abelian group is proved by Strickland \cite[Theorem 10.3]{strickland2011multicurves}, and the same argument applies to an arbitrary finite group, using the fact that for $X,Y \in \mathrm{Sp}^G$,
    \[
    (F(X, Y[\mathcal{F}_{H\not\subset}^{-1}\bigr]))^H \simeq (F(X[\mathcal{F}_{H\not\subset}^{-1}\bigr],Y[\mathcal{F}_{H\not\subset}^{-1}\bigr]))^H \simeq F(\Phi^H X, \Phi^H Y).
    \]
\end{proof}

%%-------------------------------------------------------------------------------------------------------------------------------------------------------------------------------------

%%-------------------------------------------------------------------------------------------------------------------------------------------------------------------------------------
\section{$L_{KU_G/p}S_G$ as a homotopy fiber}\label{Sec: KU/p local sphere}
In this section, let $G$ be a finite nilpotent group. For each prime $p$, let $N_p$ be the Sylow $p$-subgroup of $G$. Then $$G\cong \prod_p N_p,$$ and the projection onto the $p$-factor gives a canonical quotient map $\alpha_p:G\to N_p$. This induces the inflation functor $\operatorname{Inf}_{N_p}^G=\alpha_p^*:\mathrm{Sp}^{N_p}\to \mathrm{Sp}^G$. We identify $L_{KU_G/p}S_G$ as a homotopy fiber in \Cref{thm:KUG local spectrum first case}, which is related to the Adams operation $\psi^g$ on $KO_{N_p}$.

\begin{lemma}\label{lem:bousfield class of KU/p 1}
    Let $Cyc$ be the family of all cyclic subgroups of $G$. For any prime $p$, $KU_G/p$ is Bousfield equivalent to $ECyc_+ \wedge \operatorname{Inf}_{N_p}^G KU_{N_p}/p$.
\end{lemma}
\begin{proof}
    By \cite[Proposition 3.2]{hill2019equivariant}, it suffices to show that for any subgroup $H\subset G$, $\langle \Phi^H KU_G/p\rangle = \langle \Phi^H (ECyc_+\wedge \operatorname{Inf}_{N_p}^G KU_{N_p}/p)\rangle$. By \Cref{prop:geometric fixed KU} we have 
    \[
        \langle \Phi^H KU_G/p\rangle=\begin{cases}
            \langle KU/p\rangle  & H\in Cyc, \text{ and } p\nmid |H|.\\
            \langle \ast \rangle & \text{otherwise}.
        \end{cases}
    \]
    On the other hand, since $G$ is nilpotent, there is a subgroup $N\subset G$ such that $G\cong N_p\times N$. For any subgroup $H\subset G$, we have the following diagram of groups
    \[
    \begin{tikzcd}
        H\arrow[r, hook] & N\times (H\cap N_p) \arrow[r, twoheadrightarrow]\arrow[d, hook] & H\cap N_p\arrow[d, hook]\\
                &   G\arrow[r, twoheadrightarrow] & N_p   .
    \end{tikzcd}
    \]
    Therefore, 
    \[
    \begin{aligned}
        \Phi^H \operatorname{Inf}_{N_p}^G KU_{N_p}/p & \simeq \Phi^H \operatorname{Inf}_{H\cap N_p}^{N\times (H\cap N_p)}\operatorname{Res}_{H\cap N_p}^{N_p}KU_{N_p}/p\\
        &\simeq \Phi^H \operatorname{Inf}_{H\cap N_p}^{H}KU_{H\cap N_p}/p.
    \end{aligned}
    \]

    If $H\cap N_p\neq \{e\}$, $\Phi^{H\cap N_p} KU_{H\cap N_p}/p\simeq \ast$ since $p$ is invertible in $\Phi^{H\cap N_p}KU_{H\cap N_p}/p$. The group homomorphism $H\cong (H\cap N)\times (H\cap N_p)\to H\cap N_p$ induces a map of ring spectra
    \[\ast \simeq \Phi^{H\cap N_p} KU_{H\cap N_p}/p\to \Phi^H \operatorname{Inf}_{H\cap N_p}^H KU_{H\cap N_p}/p,\] 
    thus $\Phi^H\operatorname{Inf}_{H\cap N_p}^H KU_{H\cap N_p}/p\simeq \ast$.
    
    If $H\cap N_p=\{e\}$, then $\Phi^H\operatorname{Inf}_{N_p}^G KU_{N_p}/p\simeq \Phi^H\operatorname{Inf}_e^H KU/p\simeq KU/p$, and
    \[
    \Phi^H(ECyc_+\wedge \operatorname{Inf}_{N_p}^G KU_{N_p}/p)\simeq \begin{cases}
        KU/p & H\in Cyc, \text{ and } H\cap N_p=\{e\}.\\
        \ast & \text{otherwise.}
    \end{cases}
    \]
    Note that for the Sylow $p$-subgroup $N_p$, $H\cap N_p=e$ if and only if $p\nmid |H|$. Therefore, $\langle KU_G/p\rangle=\langle  ECyc_+\wedge \operatorname{Inf}_{N_p}^G KU_{N_p}/p \rangle$.
\end{proof}

\begin{remark}
     Let $N$ be a normal subgroup of $G$. The $N$-fixed point $(\widetilde{E\mathcal{F}_{G\not\subset}})^{N}\in \mathcal{S}^{G/N}$ can be regarded as a model for $\widetilde{E\mathcal{F}_{G/N \not\subset}}$. This gives a canonical map of $G$-spaces 
    $$i:\operatorname{Inf}_{G/N}^G \widetilde{E\mathcal{F}_{G/N \not\subset}} \to \widetilde{E \mathcal{F}_{G\not\subset}}$$
    adjoint to the identity of $\widetilde{E\mathcal{F}_{G/N \not\subset}}$.
    For any $G/N$-ring spectrum $E$, we can define the map of ring spectra $\Phi^{G/N}E\to \Phi^{G}\operatorname{Inf}_{G/N}^G E$ as the composition
    \[
    \begin{aligned}
        \Phi^{G/N}E&=(\widetilde{E\mathcal{F}_{G/N\not\subset}}\wedge E)^{G/N}\to (\widetilde{E\mathcal{F}_{G/N\not\subset}}\wedge (\operatorname{Inf}_{G/N}^G E)^N)^{G/N}\\
        &\simeq (\operatorname{Inf}_{G/N}^G \widetilde{E\mathcal{F}_{G/N\not\subset}} \wedge \operatorname{Inf}_{G/N}^G E)^G\\
        &\overset{i}{\to} (\widetilde{E \mathcal{F}_{G\not\subset}}\wedge \operatorname{Inf}_{G/N}^G E)^G=\Phi^G(\operatorname{Inf}_{G/N}^G E).
    \end{aligned}
    \]
    This construction gives the map 
    \[\Phi^{H\cap N_p} KU_{H\cap N_p}/p\to \Phi^H \operatorname{Inf}_{H\cap N_p}^H KU_{H\cap N_p}/p, \]
    which is used in the proof of \Cref{lem:bousfield class of KU/p 1}.
\end{remark}

\begin{lemma}\label{lem:bousfield class of ko}
    $KU_G$ is Bousfield equivalent to $KO_G$. 
\end{lemma}
\begin{proof}
   The proof follows the proof of the nonequivariant case in \cite[Theorem 8.4]{Ravenel1984LocalizationWR}. The equivariant Wood theorem \cite[Theorem 9.8]{mathew2017nilpotence} says that there is a cofibration 
    \[
    \Sigma KO_G \overset{\eta}{\longrightarrow} KO_G \longrightarrow KU_G
    \]
    where $\eta\in \pi_1(S)$ is the Hopf element. Then $KO_G\wedge X\simeq \ast$ implies $KU_G\wedge X\simeq \ast$. Conversely, if $KU_G\wedge X\simeq \ast$, then multiplication by $\eta$ induces an isomorphism on $(KO_G)_*(X)$. Since $\eta$ is nilpotent, it follows that $(KO_G)_*(X)=0$. 
\end{proof}

\begin{lemma}\label{lem:geometric fixed KO}
    For any subgroup $H\subset G$, $\Phi^H KU_G/p\simeq \ast$ if and only if $\Phi^H KO_G/p\simeq \ast$. Furthermore, 
    \[
    \langle KU_G/p\rangle=\langle  ECyc_+\wedge \operatorname{Inf}_{N_p}^G KO_{N_p}/p \rangle.
    \]
\end{lemma}
\begin{proof}
    Apply the geometric fixed point functor $\Phi^H$, we have a cofibration
    \[
    \Phi^H KO_G/p\wedge S^1\overset{1\wedge\eta}{\longrightarrow} \Phi^H KO_G/p \wedge S^0\longrightarrow \Phi^H KU_G/p.
    \]
    By the same argument in \Cref{lem:bousfield class of ko}, since $\eta$ is nilpotent, $\Phi^H KU_G/p\simeq \ast$ if and only if $\Phi^H KO_G/p\simeq \ast$. 

    Therefore, 
    \[
    \Phi^H(ECyc_+\wedge \operatorname{Inf}_{N_p}^G KO_{N_p}/p)\simeq \begin{cases}
        KO/p & H\in Cyc, \text{ and } H\cap N_p=\{e\}.\\
        \ast & \text{otherwise.}
    \end{cases}
    \]
    Since $\langle KU/p\rangle=\langle  KO/p \rangle$, we have $\langle KU_G/p\rangle=\langle  ECyc_+\wedge \operatorname{Inf}_{N_p}^G KO_{N_p}/p \rangle$.
\end{proof}

\begin{proposition}\label{thm:KUG local spectrum first case}
    Let $G$ be a finite nilpotent group, and let $Cyc$ be the family of all cyclic subgroups of $G$. For any prime $p$, let $N_p$ be the Sylow $p$-subgroup of $G$, and let $g$ be a topological generator of $\mathbb{Z}_p^{\times}/\{\pm 1\}$. Then for any finite $G$-spectrum $X$, there is a fiber sequence 
    \[
    L_{KU_G/p}X \longrightarrow (ECyc_+\wedge\ \operatorname{Inf}_{N_p}^G KO_{N_p}\wedge X)_p^{\wedge}\overset{\psi^g-1}{\longrightarrow} (ECyc_+\wedge\ \operatorname{Inf}_{N_p}^G KO_{N_p}\wedge X)_p^{\wedge} .
    \]
    
    When $p$ is odd, the $KO_{N_p}$ appearing in the fiber sequence above can be replaced by $KU_{N_p}$. In this case, $g=(\zeta_{p-1},p+1)$ is a topological generator of $\mathbb{Z}_p^{\times}$, where $\zeta_{p-1}$ is a primitive $(p-1)$-th root of unity.
\end{proposition}
\begin{proof}
    Let $I=ECyc_+\wedge \operatorname{Inf}_{N_p}^G KO_{N_p}$. We first claim that $\psi^g:I_p^\wedge \to I_p^\wedge$ is a map of ring spectra and let $F_G$ be the fiber of $((\psi^g - 1)\wedge X)_p^\wedge$. Then by \Cref{prop:mod p local}, $(I\wedge X)_p^{\wedge}$ is $I/p$-local; hence it is $KU_G/p$-local by \Cref{lem:bousfield class of ko}. It follows that $F_G$ is $KU_G/p$ local. 
    
    Note that 
    \[
    S_G\longrightarrow I_p^\wedge \overset{\psi^g-1}{\longrightarrow} I_p^\wedge
    \]
    is trivial. Smashing with $X$ and $p$-completing, we see that the composite 
    \[
    X_p^\wedge \longrightarrow (I\wedge X)_p^\wedge \overset{\psi^g-1}{\longrightarrow} (I\wedge X)_p^\wedge
    \]
    is also trivial, which induces a map $\iota: X_p^\wedge \to F_G$. It suffices to show that $\iota$ is a $I/p$-equivalence, or equivalently, 
    \[ 
    f_H:\Phi^H(I/p\wedge \iota): \Phi^H (I/p\wedge X_p^\wedge)\to \Phi^H (I/p\wedge F_G)
    \] 
    is an equivalence for every subgroup $H\subset G$. 
    By \Cref{prop:geometric fixed KU} and \Cref{lem:geometric fixed KO}, $\Phi^H \operatorname{Inf}_{N_p}^G KO_{N_p}/p\simeq \ast$ if $H$ is not cyclic or $H\cap N_p$ is nontrivial, and $f_H$ is an equivalence between trivial spectra. If $H$ is cyclic and $H\cap N_p=\{e\}$, then we have
    \[
    \Phi^H I/p\simeq \Phi^H\operatorname{Inf}_e^H KO/p\simeq KO/p.
    \]
    By \Cref{prop:q local geometric fixed points}, 
    \[
    \begin{aligned}
        \Phi^H (I/p\wedge F_G) & \simeq \Phi^H I/p \wedge \operatorname{fib}(\Phi^H(\psi^g-1))\\
        &\simeq KO/p \wedge \operatorname{fib}((\Phi^H (I\wedge X))_p^\wedge\longrightarrow (\Phi^H(I\wedge X))_p^\wedge)\\
        &\simeq KO/p \wedge \operatorname{fib}((KO\wedge \Phi^H X)_p^\wedge\overset{\psi^g-1}{\longrightarrow} (KO\wedge \Phi^H X)_p^\wedge)\\
        &\simeq KO/p \wedge L_{KU/p}\Phi^H X \simeq \Phi^H(I/p\wedge X).
    \end{aligned}
    \]
    Here the first equivalence follows from the fact that $\Phi^H$ preserves fiber sequences and smash products.  

    It remains to prove the claim. For any prime $q\neq p$, let $N_q$ denote the Sylow $q$-subgroup of $G$, and set $N:=\prod_{q\neq p}N_q$.
    Let $Cyc^{N}$ denote the family of cyclic subgroups of $N$, and let $ECyc^N$ be the associated $N$-space. Then $G\cong N\times N_p$, and
    \[
    I_p^\wedge \simeq
    \bigl(\operatorname{Inf}_N^G ECyc^N_+
    \wedge
    \operatorname{Inf}_{N_p}^G KO_{N_p}\bigr)_p^\wedge.
    \]
    Since $p\nmid |N|$, the $N$-spectrum $(ECyc^N_+)_p^\wedge$ is a direct summand of $(S_N)_p^\wedge$. In particular, it is an $N$-equivariant commutative ring spectrum.
    
    Recall that for any $k\in\mathbb{Z}$, there is an Adams operation $\psi^k$ on equivariant $K$-theory of a $G$-space defined in \cite{atiyah1969group}. Hirata-Kono \cite[Theorem3.1]{hirata1982bott} shows that the Adams operation $\psi^k$ induces a stable operation after inverting $k$ if and only if $(k,|G|)=1$. Then for a chosen generator $g\in \mathbb Z_p^\times$, the Adams operation induces a map of ring spectra 
    \[
    \widetilde{\psi^k}: (KO_{N_p})_p^\wedge\to (KO_{N_p})_{p}^\wedge
    \]
    for any $p$-group $N_p$ and $k\in \mathbb{Z}_p^\times/\{\pm 1\}$, and 
    $$\psi^g=(\operatorname{Inf}_N^G ECyc^N_+\wedge\ \operatorname{Inf}_{N_p}^G \widetilde{\psi^g})_p^{\wedge}: I_p^\wedge \to I_p^\wedge $$ 
    is a map of ring spectra since $(\operatorname{Inf}_N^G ECyc^N_+\wedge\ \operatorname{Inf}_{N_p}^G -)_p^{\wedge}$ is a monoidal functor.
\end{proof}

\begin{remark}\label{rmk:p group}
    When $G=N_p$ is a $p$-group, $ECyc_+\wedge KO_{N_p}$ is $S_{N_p}/p$-equivalent to $KO_{N_p}$, so we have 
    \[
    (ECyc_+\wedge KO_{N_p})_p^{\wedge}\simeq (KO_{N_p})_p^\wedge.
    \]
    In this case, the fiber sequence in \Cref{thm:KUG local spectrum first case} agrees with the fiber sequences in \cite[A.4.13]{balderrama2024total} and \cite[Proposition 6.3]{bonventre2022ku_g}.
\end{remark}

\begin{corollary}\label{cor:from Np equivariant to G equivariant}
    For any prime $p$,   
    $ L_{KU_G/p}S_G\simeq (ECyc_+\wedge \operatorname{Inf}_{N_p}^G L_{KU_{N_p}/p}S_{N_p})_p^\wedge. $
\end{corollary}
\begin{proof}
    Since $\operatorname{Inf}_{N_p}^G KO_{N_p}/p\simeq \operatorname{Inf}_{N_p}^G (KO_{N_p})_p^\wedge/p$, there is an equivalence of $G$-spectra 
    \[ 
    f:(ECyc_+\wedge \operatorname{Inf}_{N_p}^G KO_{N_p})_p^\wedge \simeq (ECyc_+\wedge \operatorname{Inf}_{N_p}^G (KO_{N_p})_p^\wedge)_p^\wedge.
    \]
    So we can rewrite the fiber sequence in \Cref{thm:KUG local spectrum first case} by applying the functor
    $(ECyc_+\wedge \operatorname{Inf}_{N_p}^G(-))_p^\wedge$
    to the fiber sequence
    \[
    L_{KU_{N_p}/p}S_{N_p}\to (KO_{N_p})_p^\wedge \to (KO_{N_p})_p^\wedge .
    \]
\end{proof}
\begin{remark}
    When $N_p=e$ for some prime $p$, let 
    \[
    F_G=ECyc_+\wedge \operatorname{Inf}_{e}^G L_{KU/p}S.
    \]
    For any subgroup $H\subset G$, $\Phi^H F_G$ is either trivial or $L_{KU/p}S$, thus $\Phi^H F_G$ is $p$-complete for all $H\subset G$. Since $p\nmid |G|$, by \Cref{prop:q local geometric fixed points}, there is an equivalence of non-equivariant spectra
    \[
    \Phi^H F_G\simeq (\Phi^H F_G)_p^\wedge \simeq \Phi^H (F_G)_p^\wedge,
    \] 
    i.e. $F_G\simeq (F_G)_p^\wedge$ is $p$-complete. 
    It follows that for $p\nmid |G|$,
    \[
    L_{KU_G/p}S_G\simeq (F_G)_p^\wedge\simeq ECyc_+\wedge \operatorname{Inf}_{e}^G L_{KU/p}S,
    \]
    which agrees with the result of \cite[Proposition 8.5]{bonventre2022ku_g}.
\end{remark}

%%-------------------------------------------------------------------------------------------------------------------------------------------------------------------------------------
\section{The fixed points of $L_{KU_G/p }S_G$}\label{Sec:fixed point of KU/p local sphere}
In order to compute $\underline{\pi}_* L_{KU_G/p }S_G$, we need to study the fixed points of $(ECyc_+\wedge \operatorname{Inf}_{N_p}^G E)_p^\wedge$ for an $N_p$-spectrum $E$. The main result of this section is \Cref{lem:Z graded homotopy Mackey functor of ECyc inf}, which allows us to compute $\underline{\pi}_* L_{KU_G/p}S_G$ for a finite nilpotent group $G$. 

\begin{lemma}\label{lem:fixed point of suspension spectra}
    Let $H$ and $K$ be finite groups with coprime orders, let $G=H\times K$, and let $X\in \mathcal{S}^K_*$ be a pointed $K$-space. Regard $X$ as a $G$-space via the quotient map $G\to K$. Then, as an $H$-spectrum, $(\Sigma^\infty_G X)^K\simeq \operatorname{Inf}_e^H (\Sigma_K^\infty X)^K$. 
\end{lemma}
\begin{proof}
    Since $(|H|,|K|)=1$, for any subgroup $N\subset G$, the Weyl group satisfies $W_G N\cong W_K L_1\times W_H L_2$, where $L_1=N\cap K$ and $L_2=N\cap H$. By tom-Dieck splitting, 
    \[
    \begin{aligned}
        (\Sigma^\infty_{G} X)^{G} & \simeq \bigoplus_{N\subset G}\Sigma^\infty EW_GN_+\wedge_{W_G N} X^N \\
        & \simeq \bigoplus_{L_1\subset K, L_2\subset H} \Sigma^\infty ((E W_H L_2)_+\wedge (EW_K L_1)_+)\wedge_{W_H L_2\times W_K L_1} X^{L_1} \\
        & \simeq \bigoplus_{L_1\subset K, L_2\subset H} \Sigma^\infty (B W_H L_2)_+\wedge ((EW_K L_1)_+\wedge_{W_K L_1} X^{L_1}) \\
        & \simeq \bigoplus_{L_2\subset H}\Sigma^\infty ((BW_H L_2)_+\wedge (\bigoplus_{L_1\subset K} (EW_K L_1)_+\wedge_{W_K L_1} X^{L_1}))\\
        & \simeq \bigoplus_{L_2\subset H}(BW_H L_2)_+\wedge (\Sigma_K^\infty X)^K\simeq (S_H)^H\wedge (\Sigma_K^\infty X)^K.
    \end{aligned}
    \]
    There is a canonical inclusion
    \[
    g: (\Sigma^\infty_K X)^K \hookrightarrow (S_H)^H\wedge (\Sigma_K^\infty X)^K \simeq (\Sigma^\infty_G X)^G.
    \]
    Let 
    \[ 
    f:\operatorname{Inf}_e^H (\Sigma_K^\infty X)^K\to (\Sigma^\infty_G X)^K
    \]
    be the map of $H$-spectra adjoint to $g$. We can show that $f$ is an $H$-equivalence. Indeed, since $S_H=\operatorname{Inf}_e^H \Sigma S^0$, the $H$-fixed point $f^H$ is precisely the equivalence 
    \[
    \begin{aligned}
        (\operatorname{Inf}_e^H (\Sigma_K^\infty X)^K)^{H}&\simeq (S_{H}\wedge  \operatorname{Inf}_e^H (\Sigma_K^\infty X)^K)^{H} \\
        & \simeq (S_{H})^{H}\wedge (\Sigma^\infty_K X)^K \simeq (\Sigma^\infty_{H\times K} X)^{H\times K},
    \end{aligned}
    \]
    and the same argument shows that $f^L\in \mathrm{Sp}$ is an equivalence for every subgroup $L\subset H$.
\end{proof}

Throughout the rest of this section, let $G$ be a finite nilpotent group with Sylow $p$-subgroup $N_p$, and let $N$ denote the product of the Sylow $q$-subgroups of $G$ for $q\neq p$. Then we have $G=N_p\times N$, and $(|N_p|, |N|)=1$.
\begin{lemma}\label{lem:fixed point of ECyc and inflation}
    For any subgroup $H\subset G$, let $P=H\cap N_p$ and $L=H\cap N$, then for any $N_p$-spectrum $E$,
    \[ 
    (ECyc_+\wedge \operatorname{Inf}_{N_p}^G E)^H\simeq (ECyc^P_+ \wedge \operatorname{Res}_P^{N_p} E)^P \wedge (\bigvee_{(T)\in Cyc^L/L}BW_L T_+).
    \]
    Here $Cyc^P$ (resp., $Cyc^L$) on the right-hand side of the equivalence is the family of all cyclic subgroups of $P$ (resp., $L$), $(T)$ is the conjugacy class of $T$.
\end{lemma}
\begin{proof}
   For any $K\subset G$, let $Cyc^K$ be the family of all cyclic subgroups of $K$, then we have 
   \[
   \operatorname{Res}_K^G ECyc_+\simeq ECyc_+^K.
   \] 
   Since $p\nmid |N|$, for any cyclic subgroup $P_0\subset N_p$ and $L_0\subset N$, $P_0\times L_0$ is also a cyclic subgroup. By the definition of $ECyc$, we have
   \[
   ECyc_+\simeq ECyc^{N_p}_+\wedge ECyc^N_+.
   \]
    Since $H=P\times L$, there is an equivalence of spectra $X^H\simeq (X^L)^P$ for any $G$-spectrum $X$. 
    
    Consider the $L$ fixed point of $ECyc_+\wedge \operatorname{Inf}_{N_p}^G E$ as a $P$-spectrum, we have
    \[
    \begin{aligned}
        (ECyc_+\wedge \operatorname{Inf}_{N_p}^G E)^L & \simeq (ECyc^L_+\wedge \operatorname{Inf}_{P}^{H} (ECyc^P_+\wedge \operatorname{Res}_{P}^{N_p}E))^L \\
        & \simeq (ECyc^P_+\wedge \operatorname{Res}_{P}^{N_p}E) \wedge (\Sigma_H^\infty ECyc_+^L)^L.
    \end{aligned}
    \] 
    By \Cref{lem:fixed point of suspension spectra}, 
    \[
    (\Sigma_H^\infty ECyc_+^L)^L\simeq \operatorname{Inf}_e^{P}(\Sigma_L^\infty ECyc_+^L)^L,
    \] 
    thus
    \[
    \begin{aligned}
        (ECyc_+\wedge \operatorname{Inf}_{N_p}^G E)^H&\simeq ((ECyc_+\wedge \operatorname{Inf}_{N_p}^G E)^L)^P \\
        &\simeq ((ECyc^P_+\wedge \operatorname{Res}_{P}^{N_p}E) \wedge (\Sigma_H^\infty ECyc_+^L)^L)^P\\
        &\simeq ((ECyc^P_+\wedge \operatorname{Res}_{P}^{N_p}E) \wedge \operatorname{Inf}_e^{P}(\Sigma_L^\infty ECyc_+^L)^L)^P\\
        &\simeq (ECyc^P_+\wedge \operatorname{Res}_{P}^{N_p}E)^P \wedge (\Sigma_L^\infty ECyc_+^L)^L,
    \end{aligned}
    \]
    where the last equivalence follows from the formula 
    \[
    (\operatorname{Inf}_e^P x\cdot y)^P=x\cdot y^P, \quad \forall \ x\in \mathrm{Sp}, \ y\in \mathrm{Sp}^P.
    \]
    By tom-Dieck splitting formula, 
    \[
    (\Sigma_L^\infty ECyc_+^L)^L\simeq \bigvee_{(T)\subset L}\Sigma^\infty EW_L T_+\wedge_{W_L T} (ECyc^L_+)^T.
    \] 
    If $T\subset L$ is not cyclic, $(ECyc^L_+)^T$ is $W_L T$-equivariant contractible, and 
    \[ EW_L T_+\wedge_{W_L T} (ECyc^L_+)^T\simeq \ast.\] 
    If $T\subset L$ is cyclic, then for any subgroup $K \subset W_L T$,
    \[
    (EW_L T_+\wedge (ECyc_+^L)^T)^{K}\simeq\begin{cases}
        S^0 & K=\{e\},\\
        \ast & K\neq \{e\},
    \end{cases}
    \]
    which implies that there is an equivalence of $W_L T$-spaces $EW_L T_+\wedge (ECyc_+^L)^T\simeq EW_L T_+$, so 
    \[
    EW_L T_+\wedge_{W_L T} (ECyc^L_+)^T\simeq BW_L T_+.
    \] 
    Therefore, we have $(\Sigma_L^\infty ECyc_+^L)^L\simeq \bigvee_{(T)\in Cyc^L/L}BW_L T_+$, and 
    \[  
    (ECyc_+\wedge \operatorname{Inf}_{N_p}^G E)^H \simeq (\operatorname{Res}_P^{N_p} E\wedge ECyc_+)^P \wedge (\bigvee_{(T)\in Cyc^L/L}BW_L T_+).
    \]
\end{proof}

To state our results, we first fix some notation. Let $A(G)$ be the Burnside ring of $G$, let $R\mathbb{Q}(G)$ (resp., $RU(G)$) be the rational (resp., complex) representation ring of $G$. For any subgroup $H\subset G$, we denote by $\underline{A}_H$ the $H$-Green functor with $\underline{A}_H(H/K)=A(K)$ for each orbit $H/K$, and define $\underline{R\mathbb{Q}}_H$ and $\underline{RU}_H$ similarly. When $H=G$, we omit the subscript and write $\underline{A}$ for $\underline{A}_G$. There is a natural homomorphism of $H$-Green functors 
\[
\mathcal{R}_H: \underline{A}_H\longrightarrow \underline{R\mathbb{Q}}_H\longrightarrow \underline{RU}_H,
\] 
which sends a finite $H$-set to the free rational (complex) vector space on the underlying set. Let $\underline{J}_H=\ker \mathcal{R}_H$, which is called the Brauer relations. Let $\underline{A/J}_H := \underline{A}_H/\underline{J}_H$.

By \cite[Proposition 3.8]{szymik2013chromatic}, the ideal $J(G)$ is generated by those elements $S\in A(G)$ such that $|S^C|=0$ for every cyclic subgroup $C\subset G$, hence the number of additive generators of $A(G)/J(G)$ equals the number of conjugacy classes of cyclic subgroups of $G$. If $G$ is a $p$-group, the map $A(G)\to R\mathbb{Q}(G)$ is surjective, and $\underline{A/J}\cong \underline{R\mathbb{Q}}$. If $p\nmid |G|$, then after $p$-completion, for any subgroup $K\subset G$,
\[(\underline{A/J})_p^\wedge (G/K) \cong \bigoplus_{(H)\leq K, H \text{ is cyclic}}\mathbb{Z}_p^\wedge.\]

\begin{lemma}\label{lem:tensor product}
Let $H$ and $K$ be finite groups of relatively prime orders. Let $\underline X$ be an $H$-Mackey functor and $\underline Y$ a $K$-Mackey functor. Then for any $P\subset H$ and $Q\subset K$, the assignment
\[
(\underline X\otimes\underline Y)\bigl((H\times K)/(P\times Q)\bigr)
:=
\underline X(H/P)\otimes\underline Y(K/Q)
\]
determines canonically an $(H\times K)$-Mackey functor. Its restriction, transfer, and conjugation maps are induced factorwise from those of $\underline X$ and $\underline Y$.
\end{lemma}
\begin{proof}
    Let $Z$ be an $H\times K$-orbit. Then $Z = (H \times K)/U$ for some subgroup $U \subset H \times K$. Let $P$ and $Q$ denote the images of $U$ under the projections onto $H$ and $K$, respectively. Since $|H|$ and $|K|$ are coprime, $|U| = |P| \times |Q|$, which implies $U \cong P \times Q$ and $Z \cong H/P \times K/Q$. 
    Since $H$ acts trivially on the $K$-orbits and $K$ acts trivially on the $H$-orbits, for $P_i\subset H$ and $Q_i\le K$, $i=1,2$, there is a natural bijection 
    \[
    \hom_{H\times K} \bigl(H/P_1\times K/Q_1,H/P_2\times K/Q_2\bigr)
    \cong 
    \hom_{H}(H/P_1,H/P_2) \times \hom_{K}(K/Q_1,K/Q_2). 
    \]
    Let $\mathcal B_{H\times K}$ be the $H\times K$-Burnside category. A transitive span
    $$H/P_1\times K/Q_1 \longleftarrow U \longrightarrow H/P_2\times K/Q_2$$
    has $U\cong H/R\times K/S$ for some $R\subset H$ and $S\subset K$, and therefore determines a pair of spans
    $$H/P_1\longleftarrow H/R\longrightarrow H/P_2,\quad K/Q_1\longleftarrow K/S\longrightarrow K/Q_2.$$
    Conversely, every such pair determines a transitive $H\times K$-span. Since an arbitrary finite $H\times K$-set is a finite disjoint union of orbits, this gives a natural isomorphism
    $$\mathcal B_{H\times K}
    \bigl(H\times K /(P_1\times Q_1),H\times K/(P_2\times Q_2)\bigr)
    \cong
    \mathcal B_H(H/P_1,H/P_2)
    \otimes_{\mathbb Z}
    \mathcal B_K(K/Q_1,K/Q_2),$$
    which is compatible with composition. Consequently, the assignment
    \[
    H/P\times K/Q
    \longmapsto
    \underline X(H/P)\otimes_{\mathbb Z}\underline Y(K/Q)
    \]
    defines an additive functor on the full subcategory of $\mathcal B_{H\times K}$ spanned by the transitive $H\times K$-sets. Since every finite $H\times K$-set is a finite disjoint union of transitive $H\times K$-sets, it extends uniquely by additivity to a $H\times K$-Mackey functor.
\end{proof}

We refer to the Mackey functor $\underline{X}\otimes \underline Y$ in \Cref{lem:tensor product} as the external tensor product of $\underline X$ and $\underline Y$. Now we can compute the coefficients of $(ECyc_+ \wedge \operatorname{Inf}_{N_p}^G E)_p^\wedge$ for an $N_p$-spectrum $E$ satisfying $E/p\simeq ECyc^{N_p}_+ \wedge E/p$. 

\begin{proposition}\label{lem:Z graded homotopy Mackey functor of ECyc inf}
    Let $G=N_p\times N$ be a finite nilpotent group with Sylow $p$-subgroup $N_p$, and let $E$ be an $N_p$-spectrum such that  $E/p\simeq ECyc^{N_p}_+ \wedge E/p$. There is an isomorphism of $G$-Mackey functors
    \[ \underline{\pi}_*(ECyc_+\wedge \operatorname{Inf}_{N_p}^G E)_p^\wedge\cong \underline{\pi}_* E_p^\wedge \otimes_{\mathbb Z_{(p)}} (\underline{A/J}_N)_p^\wedge. \] 

    In particular, when $E=L_{KU_{N_p}/p}S_{N_p}$, there is an isomorphism
    \[\underline{\pi}_*(L_{KU_G/p}S_G)\cong \underline{\pi}_*L_{KU_{N_p}/p}S_{N_p}\otimes_{\mathbb{Z}_p} (\underline{A/J}_N)_p^\wedge.\]
\end{proposition}
\begin{proof}
    For any subgroup $H\subset G$, let $P=H\cap N_p$ and $L=H\cap N$, then $H=P\times L$. By \Cref{lem:fixed point of ECyc and inflation}, there is an equivalence
    \[
    (ECyc_+\wedge \operatorname{Inf}_{N_p}^G E)^H \simeq (ECyc^P_+\wedge E)^P \wedge (\bigvee_{T\in Cyc^L}BW_L T_+).
    \]
    After $p$-completion, $(BW_L T)_p^\wedge\simeq \ast$ since $p\nmid |L|$. Moreover, since $E/p\simeq E/p\wedge ECyc^{N_p}_+$, we have $E_p^\wedge\simeq (E\wedge ECyc^P_+)_p^\wedge$. Thus
    \[
    L_{S/p}(ECyc_+\wedge \operatorname{Inf}_{N_p}^G E)^H \simeq \bigvee_{T\in Cyc^L} (E^P)_p^\wedge ,
    \]
    and 
    \[
    \pi_*^H (ECyc_+\wedge \operatorname{Inf}_{N_p}^G E)_p^\wedge\cong ((A/J)(L) \otimes \pi_*^P E)_p^\wedge.
    \]

    We also need to determine the restriction and transfer homomorphism. For any $P_1\subset P_2\subset N_p$, without loss of generality, we can assume that $N=\{e\}$, then $\operatorname{Res}_{P_1}^{P_2}$ and $\operatorname{Tr}_{P_1}^{P_2}$ are inherited from those in $\underline{\pi}_* E$. For $N_1\subset N_2\subset N$, assume that $N_p=\{e\}$. Since $p\nmid |N|$, it follows by \cite[Corollary 7.12]{bonventre2022ku_g} that the restriction and transfer maps are induced by those in $(\underline{A/J}_N)_p^\wedge$ Thus 
    \[ \underline{\pi}_*(ECyc_+\wedge \operatorname{Inf}_{N_p}^G E)_p^\wedge\cong (\underline{\pi}_* E \otimes \underline{A/J}_N)_p^\wedge. \] 

    In particular, when $E=L_{KU_{N_p}/p}S_{N_p}$, \Cref{cor:from Np equivariant to G equivariant} implies that
    $L_{KU_{N_p}/p}S_{N_p}\simeq (ECyc^{N_p}_+\wedge L_{KU_{N_p}/p}S_{N_p})_p^\wedge$, and hence $L_{KU_{N_p}/p}S_{N_p}$    satisfies the assumption of this proposition. Therefore, 
    \[
    \begin{aligned}
        \underline{\pi}_*(L_{KU_G/p}S_G) & \cong \underline{\pi}_*(ECyc_+\wedge \operatorname{Inf}_{N_p}^G L_{KU_{N_p}/p}S_{N_p})_p^\wedge \\
        & \cong \underline{\pi}_*L_{KU_{N_p}/p}S_{N_p}\otimes_{\mathbb{Z}_p} (\underline{A/J}_N)_p^\wedge .
    \end{aligned}
    \]
\end{proof}

    By \Cref{rmk:p group}, $KO_{N_p}$ and $KU_{N_p}$ satisfy the assumption of \Cref{lem:Z graded homotopy Mackey functor of ECyc inf}, which leads to the following isomorphism
    \[ \underline{\pi}_*(ECyc_+\wedge \operatorname{Inf}_{N_p}^G KU_{N_p})_p^\wedge\cong (KU_*\otimes \underline{RU}_{N_p}\otimes \underline{A/J}_N)_p^\wedge. \]

\begin{remark}
    If $G$ is a finite cyclic group, $L_{KU_G/p}S_G\simeq (\operatorname{Inf}_{N_p}^G L_{KU_{N_p}/p}S_{N_p})_p^\wedge$. For any $N_p$-spectrum $E$, we can compute the $\mathbb{Z}$-graded homotopy Mackey functor of $(\operatorname{Inf}_{N_p}^G E)_p^\wedge$ by \Cref{lem:fixed point of suspension spectra} in the same way. In this case, 
    \[ (\operatorname{Inf}_{N_p}^G E)^{P\oplus L}\simeq E^P \wedge (\Sigma_L^\infty S)^L. \]
    After $p$-completion, $\underline{\pi}_* (\operatorname{Inf}_{N_p}^G E)_p^\wedge\cong \underline{\pi}_* E_p^\wedge \otimes \underline{A}_N$. This is compatible with \Cref{lem:Z graded homotopy Mackey functor of ECyc inf}, since $J(G)=0$ for every cyclic group $G$.
\end{remark}

%%-------------------------------------------------------------------------------------------------------------------------------------------------------------------------------------
\section{A computation for finite abelian $2$-groups}\label{Sec:computation for 2-group}
By \Cref{lem:Z graded homotopy Mackey functor of ECyc inf}, we need to study $\underline{\pi}_* L_{KU_{N_p}/p}S_{N_p}$ for all primes $p$. When $p$ is odd, $\underline{\pi}_*L_{KU_{N_p}/p}S_{N_p}$ is computed in \cite{carawan2023homotopy}. In this section, we compute $\underline{\pi}_*L_{KU_{N_2}/2}S_{N_2}$ via the fiber sequence
\[
L_{KU_{N_2}/2}S_{N_2} \to (KO_{N_2})_2^\wedge \xrightarrow{\ \psi^g-1\ } (KO_{N_2})_2^\wedge
\]
when $N_2$ is an abelian $2$-group. We first compute the kernel and cokernel Mackey functors of \(\psi^g-1\), and then resolve the resulting extension problems, which are determined by the maps \(\theta_{N_2}\) in \Cref{lem:theta-general} and \(\theta_{8d+1}\) in \Cref{lem:theta-8d-plus-1-type-one,lem:theta-8d-plus-1-type-two}.

Throughout this section, $N_2$ is abelian. Let $g$ be a generator of $\mathbb{Z}_2^\times/\{\pm 1\}$. Building on the study of the $\psi^g$-action on $\underline{RU}$ in \cite{carawan2023homotopy}, we obtain the following lemma.
\begin{lemma}\label{lem: the action of psi on RO}
$\psi^g$ acts on $\underline{RO}_{N_2}$ as a homomorphism of $N_2$-Green functor. For any $K\subset N_2$, $\psi^g$ acts trivially on $RO(K;\mathbb{R})$ and permutes the generators of $RO(K;\mathbb{C})$. 
\end{lemma}
\begin{proof}
    The first statement follows from the facts that $\psi^g$ acts on $\underline{RU}$ as a homomorphism of $N_2$-Green functors \cite[Lemma 3.2]{carawan2023homotopy} and that $\psi^g$ commutes with the complexification map $c\colon \underline{RO}\to \underline{RU}$.

    For any $K\subset N_2$, every generator $\tau\in RO(K;\mathbb{R})$ is $1$-dimensional since $K$ is abelian, which is a sign representation, thus $\psi^g(\tau)=\tau^g=\tau$.

    For any irreducible $x\in RO(K;\mathbb{C})$, let $X\in RU(K)$ be the irreducible complex $K$-representation whose underlying real representation is $x$, then $\psi^g(x)$ is the underlying real representation of $\psi^g(X)$, which is irreducible by \cite[Lemma 3.1]{carawan2023homotopy}.
\end{proof}

For any $K\subset N_2$, let 
$M_K=\hom(K,O(1))$
be the group of irreducible real $K$-representation of real type, and let $C_K$ be the set of all irreducible real $K$-representation of complex type.
Then 
\[ RO(K;\mathbb{R})\cong \mathbb{Z}[M_K],\quad RO(K;\mathbb{C})\cong \mathbb{Z}\{C_K\},\]
and 
\[
\pi_* KO_{K} \cong \mathbb{Z}[\eta,\alpha,u^{\pm}]/(2\eta, \eta^3, \eta\alpha, \alpha^2-4u)\{\tau:\tau\in M_K\}\bigoplus \mathbb{Z}[\beta^{\pm}]\{x:x\in C_K\}.
\]
The Adams operation $\psi^g$ on $\underline{\pi}_* KO_{K}$ is given by 
\begin{itemize}
    \item[(1)] $\psi^g(\eta)=\eta$, $\psi^g(\alpha)=g^2\alpha$, and $\psi^g(u)=g^4u$.
	\item[(2)] for any $\tau\in M_K$, $\psi^g \tau=\tau$.
	\item[(3)] for any $x\in C_K$, $\psi^g(x\beta^i)=g^i\psi^g(x)\beta^i$.
\end{itemize}
For any $k\in \mathbb{Z}$, define
    \[\underline{\ker}_2\{k\}:=\ker(\underline{\pi}_k(KO_{N_2})_2^{\wedge}\overset{\psi^g-1}{\longrightarrow}\underline{\pi}_k (KO_{N_2})_2^{\wedge}),\]
    \[ \underline{\mathrm{coker}}_2\{k\}:=\mathrm{coker}(\underline{\pi}_k (KO_{N_2})_2^{\wedge}\overset{\psi^g-1}{\longrightarrow}\underline{\pi}_k (KO_{N_2})_2^{\wedge}).\]
These Mackey functors can be computed by following the method of \cite[Proposition 3.7]{carawan2023homotopy}.  

\begin{lemma}\label{lem:psi-1}
    As $N_2$-Mackey functors, $\underline{\ker}_2\{k\}$ and $\underline{\mathrm{coker}}_2\{k\}$ are determined by the following values on orbits.
    
    (1) For $k=0$, and for any $K\subset N_2$,
        \[\underline{\ker}_2\{0\}\cong \underline{R\mathbb{Q}}_2^\wedge, \quad \underline{\mathrm{coker}}_2\{0\}(N_2/K)\cong \bigoplus_{\text{cyclic } T\subset K}\mathbb{Z}_2^\wedge.\]

    (2) For $k=8d$ and $d\neq 0$, $\underline{\ker}_2\{8d\}\cong 0$, and \[\underline{\mathrm{coker}}_2\{8d\}(N_2/K)\cong RO(K;\mathbb{R})\otimes \mathbb{Z}/2^{4+\nu_2(d)} \oplus (\bigoplus\limits_{C_{2^t}\subset K, t\geq 2} \mathbb{Z}/2^{2+t+\nu_2(d)}),\]
        where $\nu_2$ is the $2$-adic valuation.%, that is, for any $d\in \mathbb{Z}$, $d=2^{\nu_2(d)}t$ such that $2\nmid t$. 

    (3) For $k=8d+1$,
        \[
        \begin{aligned}
           \underline{\ker}_2\{8d+1\} & \cong \underline{RO(-;\mathbb{R})}\otimes \mathbb{Z}/2\{\eta u^d\},\\ 
           \underline{\mathrm{coker}}_2\{8d+1\} & \cong \underline{RO(-;\mathbb{R})}\otimes \mathbb{Z}/2\{[\eta u^d]\}.
        \end{aligned}
        \]
        Here $[x]\in \underline{\mathrm{coker}}_2\{8d+1\}$ is the equivalence class of the corresponding element $x\in \underline{\pi}_{8d+1}KO_{N_2}$.  For any $K\subset N_2$, $\underline{RO(-;\mathbb{R})}(N_2/K)=RO(K;\mathbb{R})$, the restriction and transfer maps are those in $\underline{RO}$, after quotienting out all elements in $RO(-;\mathbb{C})$.
    
    (4) For $k=8d+2$,
    \[\begin{aligned}
        \underline{\ker}_2\{8d+2\}(N_2/K) & \cong RO(K;\mathbb{R})\otimes \mathbb{Z}/2\{\eta^2 u^d\},\\
        \underline{\mathrm{coker}}_2\{8d+2\}(N_2/K) & \cong (RO(K;\mathbb{R})\otimes \mathbb{Z}/2\{\eta^2 u^d\}) \oplus (\bigoplus\limits_{C_{2^t}\subset K, t\geq 2} \mathbb{Z}/2^{t}).
    \end{aligned}
    \]

    (5) For $k=8d+4$,  $\underline{\ker}_2\{8d+4\}(N_2/K) \cong 0$, and 
    \[
       \underline{\mathrm{coker}}_2\{8d+4\}(N_2/K) \cong RO(K;\mathbb{R})/2^3 \bigoplus (\bigoplus_{C_{2^t}\subset K, t\geq 2} \mathbb{Z}/2^{t+1}).
    \]

    (6) For $k=8d+6$, $\underline{\ker}_2\{8d+6\}(N_2/K) \cong 0$, and 
    \[
        \underline{\mathrm{coker}}_2\{8d+6\}(N_2/K) \cong \bigoplus\limits_{C_{2^t}\subset K, t\geq 2} \mathbb{Z}/2^{t}.
    \] 
\end{lemma}
\begin{proof}
    Since $\psi^g$ is a homomorphism of Green functors, the restriction and transfer homomorphisms are inherited from those in $\underline{\pi}_* (KO_{N_2})_2^\wedge$, and it suffices to compute for any orbit $N_2/K$. 
    
    (1) and (2): When $k=0$, $\underline{\ker}_2\{0\}=(\underline{RO}^{\psi^g})_2^\wedge$. For any $V\in KU(K)$ such that $\psi^g V=V$, by \cite[Proposition 6.7]{bonventre2022ku_g}, the character $\chi_V$ takes values in $\mathbb{Q}$, i.e., $V$ is in the image of $R\mathbb{Q}(K)\to RU(K)$, which factors through $RO(K)$. Then 
    \[ 
    \underline{RO}^{\psi^g}\cong \underline{RU}^{\psi^g}\cong \underline{R\mathbb{Q}_\chi},
    \]
    where 
    $$\underline{R\mathbb{Q}_\chi} (N_2/K):=R\mathbb{Q}_\chi(K)=\{V\in RU(K):\chi_V \text{ takes values in } \mathbb{Q}\}.$$
    By \cite[Proposition 35]{serre1977linear}, the Schur indices for $N_2$ equal to $1$ since $N_2$ is abelian. So $\underline{R\mathbb{Q}_\chi}\cong \underline{R\mathbb{Q}}$, 
    and $\underline{\ker}_2\{0\}\cong \underline{R\mathbb{Q}}_2^\wedge$.

    By \Cref{lem: the action of psi on RO}, we can decompose $C_K$ into orbits under $\psi^g$.  Since irreducible rational $K$-representations are in bijection with the cyclic subgroups of $K$, the orbits in $C_K$ are in bijection with $\{ C_{2^t}\subset K:t\geq 2\}$, and elements in $M_K$ are  in bijection with the subgroups $C_{2^t}\subset K$ for $t=0,1$. Let $[C_{2^t}\subset K]\subset C_K$ denote the orbit corresponds to the subgroup $C_{2^t}\subset K$. Since the degree of the $2^t$-th cyclotomic polynomial 
    \[
    \Phi_{2^t}(x)=\prod_{\gcd(k,2^t)=1,\;1\leq k<2^t}\left(x-e^{2\pi i k/2^t}\right)
    \]
    is $2^{t-1}$, the orbit $[C_{2^t}\subset K]$ has $2^{t-2}$ elements.

    On every orbit $[C_{2^t}\subset K]$ with $t\geq 2$, we can choose a basis such that $\psi^g$ acts as the matrix
    \[ 
    M=
    \begin{pmatrix}
    0 & 1 &        &        &        \\
        &     0 & 1&        &        \\
        &        & \ddots& \ddots &        \\
        &        &       &      0& 1 \\
    1 &    &       &        & 0
    \end{pmatrix}.
    \]
    For $k=8d$, $\psi^g-1$ acts on $RO(K)_2^\wedge\{u^d\}$ via $g^{4d}M-I$ on every orbit, which are equivalent to diagonal matrices 
    \[
    g^{4d}M-I\simeq
    \begin{pmatrix}
    1 &   &   &   &   \\
        & 1 &   &   &   \\
        &   & \ddots &   &   \\
        &   &   & 1 &   \\
        &   &   &   & g^{2^{t-2}\cdot 4d}-1
    \end{pmatrix}
    \] 
    using a combination of row and column operations. Thus $\psi^g-1$ is injective when $d\neq 0$.

    When $d=0$, every orbit contributes a summand of $\mathbb{Z}_2^\wedge $ in cokernel, so   
    \[
    \mathrm{coker}_2\{0\}(N_2/K)\cong \bigoplus_{\text{cyclic }T\subset K} \mathbb{Z}_2^\wedge.
    \] 
    When $d\neq 0$, every orbit $[C_{2^t}\subset K]$ contributes a summand of $\mathbb{Z}_2^\wedge/g^{4d}-1$ in cokernel for $t\leq 1$, and a summand of $\mathbb{Z}_2^\wedge/g^{2^td}-1$ in cokernel for $t\geq 2$. Since $g$ is a generator of $(\mathbb{Z}_2^\wedge)^\times/\{\pm 1\}\cong 1+4\mathbb{Z}_2^\wedge$, then $1-g^{2^td}$ is a generator of $2^{2+t+\nu_2(d)}\mathbb{Z}_2^\wedge$, thus 
    \[
    \mathbb{Z}_2^\wedge/g^{2^t d}-1\cong \mathbb{Z}/2^{2+t+\nu_2(d)}, 
    \]
    and 
    \[\underline{\mathrm{coker}}_2\{8d\}(N_2/K)\cong (RO(K;\mathbb{R})\otimes \mathbb{Z}/2^{4+\nu_2(d)}) \oplus (\bigoplus\limits_{C_{2^t}\subset K, t\geq 2} \mathbb{Z}/2^{2+t+\nu_2(d)}).\]

    (3) For $k=8d+1$, $\pi_{8d+1} KO_{N_2}\cong RO(N_2;\mathbb{R})\{\eta u^d\}/2$, thus $\psi^g$ acts trivially on $\underline{\pi}_{8d+1}KO_{N_2}$. 
    
    Finally, (4)-(6) follow from the same computation as above.  
\end{proof}

\begin{remark}\label{rmk: nilpotent N_2}
    When $N_2$ is nonabelian, $RO(N_2;\mathbb{H})$ is nontrivial. By \cite[Chapter 13]{serre1977linear}, for any $V\in RO(N_2;\mathbb{H})$, the Schur index of the complexification of $V$ over $\mathbb{Q}$ equals $2$. In this case, for any orbit $N_2/K$
    \[
    \underline{\ker}_2\{0\}(N_2/K)\cong R\mathbb{Q}_\chi (K)_2^\wedge\not\cong R\mathbb{Q}(K)_2^\wedge.
    \]
    Therefore,
    $$\underline{\ker}_2\{0\}\cong (\underline{R\mathbb{Q}_{\chi}}_{N_2})_2^\wedge \not\cong (\underline{R\mathbb{Q}}_{N_2})_2^\wedge.$$
    Furthermore, for any  $V\in RO(N_2;\mathbb{R})$, the Schur index of the complexification of $V$ over $\mathbb{Q}$ equals $1$, which implies that 
    \[
    \begin{aligned}
        \underline{\ker}_2\{1\}(N_2/K)&\cong RO(K;\mathbb{R})^{\psi^g}/2 \cong RO(K;\mathbb{R})\cap R\mathbb{Q}(K)/2,\\
        \underline{\mathrm{coker}}_2\{1\}(N_2/K) &\cong \mathbb{Z}/2\otimes_{\mathbb{Z}_2}\underline{\mathrm{coker}}_2\{0\}(N_2/K)/(RO(K;\mathbb{C})\oplus RO(K;\mathbb{H})).
    \end{aligned}
    \] 
    Here $V\in R\mathbb{Q}(K)$ is regarded as a real representation via the canonical inclusion $R\mathbb{Q}(K)\to RO(K)$.

For example, if $N_2=Q_8$, then $N_2$ has four one-dimensional irreducible complex representations $\rho_i$, $1\leq i\leq 4$, and one two-dimensional irreducible complex representation $\theta$. The Schur indices of the $\rho_i$ are $1$, and the Schur index of $\theta$ is $2$. Therefore, after suitably choosing the $\rho_i$, we have 
$$R\mathbb{Q}_\chi (N_2)\cong \mathbb{Z}\{\rho_1,\rho_2,\rho_3+\rho_4,\theta\},
\qquad
R\mathbb{Q}(N_2)\cong \mathbb{Z}\{\rho_1,\rho_2,\rho_3+\rho_4,2\theta\}.$$ 
As a result, $\underline{\ker}_2\{0\}\cong (\underline{R\mathbb{Q}_{\chi}}_{Q_8})_2^\wedge$, and $$\underline{\ker}_2\{1\}(Q_8/Q_8)\cong \underline{\mathrm{coker}}_2\{1\}(Q_8/Q_8) \cong \mathbb{Z}/2\{\rho_1,\rho_2, \rho_3+\rho_4\}.$$
\end{remark}

For an abelian $2$-group $N_2$ and any $k\in \mathbb{Z}$, we can compute $\underline{\pi}_k L_{KU_{N_2}/2}S_{N_2}$ via the short exact sequence:
    \[0\longrightarrow \underline{\mathrm{coker}}_2\{k+1\}\longrightarrow \underline{\pi}_k L_{KU_{N_2}/2}S_{N_2} \longrightarrow \underline{\ker}_2\{k\}\longrightarrow 0.\]
By \Cref{lem:psi-1}, we need to solve the extension problems when $k=0$ and $k=8d+1$.

\subsection{Extension problem for $k=0$} When $k=0$, we need to study the exact sequence
\[
0\longrightarrow \underline{RO(-;\mathbb{R})}_{N_2}\{\eta\}/2\xrightarrow{i} \underline{\pi}_0 L_{KU_{N_2}/2}S_{N_2}\xrightarrow{\pi} (\underline{R\mathbb{Q}}_{N_2})_2^\wedge\longrightarrow 0.
\]
Since $N_2$ is a $2$-group, $\underline{R\mathbb{Q}}_{N_2}\cong \underline{A/J}_{N_2}$, and the Hurewicz map $$(\underline{A}_{N_2})_2^\wedge \to \underline{\pi}_0 L_{KU_{N_2}/2}S_{N_2}\xrightarrow{\pi} (\underline{R\mathbb{Q}}_{N_2})_2^\wedge$$ induces a morphism of Mackey functors
\[
  \theta_{N_2}:\underline{J}_{N_2}\longrightarrow \underline{RO(-;\mathbb{R})}_{N_2}\{\eta\}/2.
\]
For $L\subset K$, there is a natural map of $K$-spectra $\operatorname{Inf}_{K/L}^K KO_{K/L}\to KO_K$, which induces a canonical inflation map $\operatorname{Inf}_{K/L}^K L_{KU_{K/L}/2}S_{K/L} \to L_{KU_K/2}S_K$. Note that the map $\theta$ is defined by the Hurewicz map, since the Hurewicz map is compatible with transfers and inflation maps, so is $\theta$.

\begin{lemma}\label{lem:quotient-description}
With notations as above, there is a natural isomorphism of Mackey functors
\[
\underline{\pi}_0 L_{KU_{N_2}/2}S_{N_2}\cong
  \frac{(\underline{A}_{N_2})_2^\wedge \oplus \underline{RO(-;\mathbb{R})}_{N_2}\{\eta\}/2}
       {\{j-\theta_{N_2}(j):j\in (\underline{J}_{N_2})_2^\wedge\}}.
\]
\end{lemma}

\begin{proof}
    Since $\underline{A}_{N_2}$ is a representable $N_2$-Mackey functor, there is no nontrivial extension of $\underline{A}_{N_2}$ by $\underline{RO(-;\mathbb{R})}_{N_2}\{\eta\}/2$. Consider the following commutative diagram with exact rows:
\[
\begin{tikzcd}[column sep=small]
0 \arrow[r] & \underline{RO(-;\mathbb{R})}_{N_2}\{\eta\}/2 \arrow[r] \arrow[d, "\cong"'] 
&
(\underline{A}_{N_2})_2^\wedge\oplus \underline{RO(-;\mathbb{R})}_{N_2}\{\eta\}/2
\arrow[r,"p_2"] \arrow[d, "\Phi"]
& (\underline{A}_{N_2})_2^\wedge \arrow[r] \arrow[d, "\mathcal{R}_{N_2}"] \arrow[ld,"h"] & 0 
\\
0 \arrow[r] & \underline{RO(-;\mathbb{R})}_{N_2}\{\eta\}/2 \arrow[r,"i"]
&
\underline{\pi}_0 L_{KU_{N_2}/2}S_{N_2}
\arrow[r,"\pi"]
&
(\underline{R\mathbb{Q}}_{N_2})_2^\wedge
\arrow[r]
&
0 .
\end{tikzcd}
\]
Here $h$ is the Hurewicz map, $\mathcal{R}_{N_2}$ is the linearization map, and 
\[
  \Phi(p,c)=h(p)+i(c).
\]
$\Phi$ is a morphism of Mackey functors since both $h$ and $i$ are. It is surjective by the five lemma. If $\Phi(j,c)=0$, then applying $\pi$ gives
\[
  0=\pi\Phi(j,c)=\pi h(j)=\mathcal{R}_{N_2}(j),
\]
so $j\in (\underline{J}_{N_2})_2^\wedge$. By definition of $\theta_{N_2}$,
\[
  0=h(j)+i(c)=i(\theta_{N_2}(j)+c).
\]
Since $i$ is injective, $c=-\theta_{N_2}(j)$.  Thus
\[
  \ker(\Phi)=\{(j,-\theta_{N_2}(j)):j\in J_{N_2}\}.
\]
Since all maps involved are Mackey functor maps, the displayed kernel is a sub-Mackey functor and the quotient is a Mackey functor quotient.
\end{proof}

Let $V:=C_2\times C_2$, and let $A,B,C\subset V$ be the three subgroups of order two. Then $J(V)\cong \mathbb{Z}$ is freely generated by
\[
  X_V=([V/A]-1)([V/B]-1)([V/C]-1)-1.
\]
Following a helpful comment of Balderrama, we have the following lemma. 
\begin{lemma}\label{lem:case N_2=V}
   Let $\rho_V$ be the regular real $V$-representation.  Then
    \[
    \theta_V(X_V)=\eta\cdot \rho_{V} \in RO(V;\mathbb R)/2\cdot \eta.
    \]
\end{lemma}
\begin{proof}
Note that $\operatorname{Res}^V_H X_V=0$ for every proper subgroup $H \subset V$. Hence
$\theta_V(X_V)$ restricts to zero on all proper subgroups. The common kernel of the restriction maps
\[
  RO(V;\mathbb R)/2\longrightarrow \prod_{|H|=2} RO(H;\mathbb R)/2
\]
is generated by $\rho_V$. Thus $\theta_V(X_V)$ is either $0$ or $\eta\rho_{V}$. Szymik \cite[Example 5.1]{szymik2013chromatic} shows that $X_V$ has nontrivial Hurewicz image, so $\theta_V(X_V)=\eta\rho_V$.
\end{proof}

Consequently, when $N_2\cong V$, at the top orbit one has an isomorphism of $A(V)$-modules
\[
\pi_0^V L_{KU_V/2}S_V
  \cong
  \frac{A(V)^{\wedge}_2\oplus RO(V;\mathbb R)\{\eta\}/2}
       {\langle (X_V-\eta\rho_{V})\rangle}.
\]

For an arbitrary finite abelian $2$-group $N_2$, it follows from \cite[Theorem 5.3]{bartel2015brauer} that all Brauer relations of $N_2$ are $\mathbb{Z}$-linear combinations of relations lifted from subquotients isomorphic to $C_2\times C_2$. More precisely, for all
\[
  L\subset K\subset N_2,
  \qquad
  K/L\cong C_2\times C_2,
\]
the virtual $N_2$-sets
\[
  X_{K,L}=\operatorname{Tr}_K^{N_2}\operatorname{Inf}_{K/L}^{K}(X_{K/L})\in J(N_2)
\]
generate $J(N_2)$. This determines the $N_2$-Mackey functor $\underline{J}_{N_2}$.

\begin{lemma}\label{lem:theta-general}
For every such subquotient $K/L\cong C_2\times C_2$,
\[
\theta_{N_2}(X_{K,L})
  =
  \eta\cdot \operatorname{Tr}_K^{N_2}\operatorname{Inf}_{K/L}^{K}(\rho_{K/L})=
  \eta\sum_{\substack{\chi:N_2\to \{\pm1\}\\ L\subseteq \ker(\chi)}}\chi
  \in RO(N_2;\mathbb R)/2\cdot\eta.
\]
\end{lemma}
\begin{proof}
The map $\theta$ is natural for transfer and inflations. Therefore
\[
\begin{aligned}
    \theta_{N_2}\big(\operatorname{Tr}_K^N\operatorname{Inf}_{K/L}^{K}(X_{K/L})\big) & =  \operatorname{Tr}_K^{N_2}\operatorname{Inf}_{K/L}^{K}\big(\theta_{K/L}(X_{K/L})\big)\\
  & = \eta\cdot \operatorname{Tr}_K^{N_2}\operatorname{Inf}_{K/L}^{K}(\rho_{K/L}).
\end{aligned}
  \]
As a real $N_2$-representation, $\operatorname{Tr}_K^{N_2}\operatorname{Inf}_{K/L}^{K}(\rho_{K/L})$ is isomorphic to $\mathbb{R}\{N_2/L\}$.
Its image in $RO(N_2;\mathbb{R})/2$ is the sum of all real one-dimensional characters of $N_2$ that are
trivial on $L$, which gives the result.
\end{proof}

Then we can determine the $N_2$-Mackey functor $\underline{\pi}_0 L_{KU_{N_2}/2}S_{N_2}$.

\begin{proposition}\label{prop:extension of 2 group KU local sphere-degree-zero}
Let $N_2$ be a finite abelian $2$-group.  There is an isomorphism of $N_2$-Mackey functors
\[
  \underline{\pi}_0 L_{KU_{N_2}/2}S_{N_2}
  \cong
  \frac{(\underline{A}_{N_2})_2^\wedge \oplus \underline{RO(-;\mathbb{R})}_{N_2}\{\eta\}/2}
       {\{j-\theta_{N_2}(j):j\in (\underline{J}_{N_2})_2^\wedge\}}.
\]
The map $\theta_{N_2}$ is determined on the generators $X_{K,L}$ by
\[
  \theta_{N_2}(X_{K,L})
  =
  \eta\sum_{\substack{\chi:{N_2}\to\{\pm1\}\\ L\subseteq\ker(\chi)}}\chi.
\]
\end{proposition}
\begin{proof}
    If a relation $j\in J(N_2)$ is written as an integral combination
\[
  j=\sum_{\substack{K,L\\ K/L\cong C_2\times C_2}} a_{K,L}X_{K,L},
\]
then
\[
\theta_{N_2}(j)  =  \eta\sum_{\substack{K,L\\ K/L\cong C_2\times C_2}} (a_{K,L}\bmod 2) \sum_{\substack{\chi:{N_2}\to\{\pm1\}\\ L\subseteq\ker(\chi)}}\chi .
\]
The expression of $j$ in terms of the generators $X_{K,L}$ is not canonical, but the resulting value
of $\theta_{N_2}(j)$ is canonical because $\theta_{N_2}$ is defined as the Hurewicz image of $j$.

At a general orbit $N_2/H$, the Mackey functor map is obtained by replacing $N_2$ with $H$:
\[
  \theta_{N_2}(N_2/H):J(H)\longrightarrow RO(H;\mathbb R)/2\cdot\eta,
\]
and the same formula applies to subquotients $L\subset K\subset H$ with $K/L\cong C_2\times C_2$. 
\end{proof}

\subsection{Extension problem for $k=8d+1$}
When $k=8d+1$, the short exact sequence has the form
\[
0\longrightarrow \underline{\mathrm{coker}}_2\{8d+2\} \xrightarrow{i} \underline{\pi}_{8d+1} L_{KU_{N_2}/2}S_{N_2}\xrightarrow{\pi} \underline{RO(-;\mathbb{R})}_{N_2}\{\eta u^d\}/2\longrightarrow 0.
\]
As in the non-equivariant case, after evaluating at $N_2/K$, the sequence splits as abelian groups.

\begin{lemma}\label{lem: group structure}
    For any subgroup $K\subset N_2$, there is an isomorphism of abelian groups $\pi_{8d+1}^K L_{KU_{N_2}/2}S_{N_2} \cong \underline{\mathrm{coker}}_2\{8d+2\}(N_2/K) \oplus \underline{\ker}_2\{8d+1\}(N_2/K)$.
\end{lemma}
\begin{proof}
    The exact sequence has the form
    \[0\to  \underline{\mathrm{coker}}_2\{8d+2\}(N_2/K) \to \pi_k^K L_{KU_{N_2}/2}S_{N_2} \to RO(K;\mathbb{R})\{\eta u^d\}/2\to 0,\]
    where all elements in $\underline{\mathrm{coker}}_2\{8d+2\}(N_2/K)$ are torsion. When $d=0$, $\eta\in \pi_1 KO$ is the Hurewicz image of the Hopf element. The composite homomorphism
    \[
        \pi_1^K S  \to \pi_1^K L_{KU_{N_2/2}}S_{N_2} \to \pi_1^K (KO_{N_2})_2^\wedge
    \]
    sends $\eta$ to $ \epsilon\eta\in \pi_1^K KO_{N_2}$ and $\operatorname{Tr}_{\ker \tau}^K(1) \eta$ to $(\epsilon+\tau)\eta \in\pi_1^K KO_{N_2}$.
    Then $\tau\eta\in RO(K;\mathbb{R})\{\eta\}/2$ lifts to an element in $\pi_1^K L_{KU_{N_2}/2}S$ of order $2$. Thus the extension problem when $k=1$ is trivial. For general $d\in \mathbb{Z}$, consider the element $\mu_d\in \pi_{8d+1}L_{K(1)}S$ of order $2$ which maps to $\eta u^d \in \pi_{8d+1} KO$, the same argument shows that all the extension problems are trivial.
\end{proof}

This sequence splits after evaluating at each orbit $N_2/K$, as a sequence
of abelian groups. However, this does not imply that the sequence splits as Mackey functors. It remains to determine the $\underline{A}$-module structure of $\underline{\pi}_k L_{KU_{N_2}/2}S_{N_2}$. Consider the map
\[
\mathcal{I}_{N_2}:\underline A_{N_2} \overset{\mathcal{R}_{N_2}}{\longrightarrow} \underline{RO}_{N_2}\longrightarrow \underline{RO(-;\mathbb R)}_{N_2}/2.
\]
Fix a lift $\eta_{N_2} u^d =\mu_d \in \underline{\pi}_{8d+1} L_{KU_{N_2}/2}S_{N_2}$ of $\eta u^d\in \pi_{8d+1}KO_{N_2}$. The action of the Burnside Mackey functor on the class $\eta_{N_2} u^d$ lifts the map $\mathcal{I}_{N_2}\otimes \mathbb{Z}\{\eta u^d\}$ to a morphism of $N_2$-Mackey functors
\[
\widetilde h:\underline A_{N_2}\longrightarrow \underline{\pi}_{8d+1} L_{KU_{N_2}/2}S_{N_2},
\]
which is compatible with the inflations. On the orbit $N_2/K$, it is given by
\[
\widetilde h_K([K/H]) = \operatorname{Tr}_H^K\operatorname{Res}_H^{N_2}(\eta_{N_2}u^d)
\]
for $H\subset K$. Since $2\eta=0$, this morphism factors through $\underline A_{N_2}/2$. Let
\[
\underline{I}_{8d+1}:=\ker\big(\underline A_{N_2}/2
\to
\underline{RO(-;\mathbb R)}_{N_2}\{\eta u^d\}/2\big)
\]
denote its kernel, then $\widetilde{h}$ induces
\[
\theta_{8d+1}:
\underline{I}_{8d+1}
\longrightarrow
\underline{\mathrm{coker}}_2\{8d+2\}.
\]
Since $\widetilde h$ is compatible with transfers and inflation maps, so is the induced map $\theta_{8d+1}$.
With this notation, similarly to \Cref{lem:quotient-description}, the extension in degree $8d+1$ is given by
\[
\underline{\pi}_{8d+1}L_{KU_{N_2}/2}S_{N_2}
\cong
\frac{
\underline A_{N_2}/2\oplus \underline{\mathrm{coker}}_2\{8d+2\}
}{
\left\{
r-\theta_{8d+1}(r):
r\in \underline{I}_{8d+1}
\right\}
}.
\]
It remains to determine $\theta_{8d+1}$.

\begin{lemma}\label{lem:I-8d-plus-1-generators}
Let $K\subset N_2$.  The group
\[
I_{8d+1}(N_2/K) = \ker\left( A(K)/2 \longrightarrow RO(K;\mathbb R)\{\eta u^d\}/2 \right)
\]
is generated by the following two types of elements.

\begin{enumerate}
\item[(i)] Let $2K=\{2x:x\in K\}\subset K$. If $H,H'\subset K$ satisfy
\[
H+2K=H'+2K,
\]
then $D_{H,H'}=[K/H]+[K/H']$ lies in $I_{8d+1}(N_2/K)$.

\item[(ii)] Let $ 2K\subseteq L\subset T\subset K$ with $T/L\cong C_2\times C_2$.
Let $M_1,M_2,M_3$ be the three intermediate subgroups between $L$ and
$T$.  Then
\[
B_{T,L} = [K/L]+[K/M_1]+[K/M_2]+[K/M_3]
\]
lies in $I_{8d+1}(N_2/K)$.
\end{enumerate}
Moreover, the elements of types $(i)$ and $(ii)$ generate
$I_{8d+1}(N_2/K)$ as an $\mathbb F_2$-vector space.
\end{lemma}

\begin{proof}
The map
\[
A(K)/2\longrightarrow RO(K;\mathbb R)\{\eta u^d\}/2
\]
only depends on the image of a subgroup in $K/2K$.  Indeed, for
$H\subset K$,
\[
[K/H]\longmapsto
\left(
\sum_{\substack{\chi:K\to\{\pm1\}\\ H\subseteq \ker(\chi)}}\chi
\right)\eta u^d,
\]
and the condition $H\subseteq \ker(\chi)$ is equivalent to
$H+2K\subseteq \ker(\chi)$.  Hence the elements $D_{H,H'}$ are in the
kernel, and after quotienting by these relations we may identify the
source of the map with $A(K/2K)/2$.

Put $\overline{K}=K/2K$.  It remains to determine the kernel of
\[
A(\overline{K})/2 \longrightarrow RO(\overline{K};\mathbb{R})/2.
\]
For a subgroup $U\subset \overline{K}$, the $\overline{K}$-set
$\overline{K}/U$ maps to
\[
\sum_{\substack{\chi:\overline{K}\to\{\pm1\}\\ U\subseteq\ker(\chi)}}\chi.
\]
If $U\subset W\subset \overline{K}$ and $W/U\cong C_2\times C_2$, with
intermediate subgroups $U_1,U_2,U_3$, then
\[
\overline{B}_{W,U}:=[\overline{K}/U]+[\overline{K}/U_1]+[\overline{K}/U_2]+[\overline{K}/U_3]
\]
maps to zero.  Indeed, a character trivial on $U$ is either trivial on
$W$, in which case it is counted four times, or has kernel one of the
three intermediate subgroups, in which case it is counted twice.

Let $R\subset A(\overline K)/2$ be the subgroup generated by the elements $\overline{B}_{W,U}$. Then the quotient $A(\overline K)/(2,R)$ is generated by the class $[\overline K/\overline K]$ and the classes $[\overline K/H]$ with $\overline K/H\cong C_2$. Their images are $1+\chi$ for all characters $\chi$ of $\overline K$, and these elements form a basis of $RO(\overline K;\mathbb R)/2$. Therefore there are no further relations. Pulling this description back along $K\to K/2K$ gives the stated generators.
\end{proof}

\begin{lemma}\label{lem:theta-8d-plus-1-type-one}
    Let $K\subset N_2$, and let $H,H'\subset K$ satisfy $H+2K=H'+2K$. For the first-type generator $D_{H,H'}$ in \Cref{lem:I-8d-plus-1-generators}, consider the projection onto the summand of $\underline{\mathrm{coker}}_2\{2\}$ indexed by a subgroup $M\cong C_{2^t}\subset K$. Let $\rho_M$ be the irreducible rational $K$-representation corresponding to $M$. Then
    \[
    \operatorname{Pr}_M\theta_{8d+1}(D_{H,H'}) = \begin{cases} 
        2^{t-1}, & t\geq 2 \text{ and exactly one of } H, H' \text{ is }\\
        & \text{contained in } \ker \rho_M,\\ 
        0, & \text{otherwise.}
    \end{cases}
    \]
\end{lemma}
\begin{proof}
    First we consider the case $K\cong C_4$, the only generator is 
    \[
    D_{C_2,\{e\}} = [C_4/C_2] + [C_4/\{e\}].
    \]
    Let $\rho$ be the faithful two-dimensional rotation $C_4$-representation. The computation of $\underline{\pi}_1 L_{KU_{C_4}/2}S_{C_4}$ in \Cref{sec:C4} shows that $$\rho\cdot \eta \neq 0\in \pi_1^{C_4}L_{KU_{C_4}/2}S_{C_4},$$ 
    and 
    \[\theta_{8d+1}(D_{C_2,\{e\}})=D_{C_2,\{e\}} \cdot \eta u^d = (\operatorname{Tr}_{e}^{C_4}(1)+\operatorname{Tr}_{C_2}^{C_4}(1))\eta u^d = [2\rho \beta u^d].\] 

    In general, let $K\subset N_2$ and $H,H'\subset K$ such that $H+2K=H'+2K$, we can reduce the computation to $D_{C_2,\{e\}}$. Let $J = H +2 K $, then in $A(K)/2$ we have
    \[
    [K/H] + [K/H'] = ([K/H] + [K/J]) + ([K/H'] + [K/J]).
    \]
    Hence it suffices to study the generators of the form
    \[
    [K/L] + [K/(L\langle x^2 \rangle)]
    \]
    for some subgroup $L\subset K$ and $x\in K\backslash L$.
    Let $T = L \langle x \rangle$. Then
    \[
    [K/L] + [K/L \langle x^2 \rangle] = \operatorname{Ind}_T^K ([T/L] + [T/L \langle x^2 \rangle]).
    \]
    Now $T/L$ is a cyclic $2$-group. Suppose $T/L \cong C_{2^m}$.
    In $A(C_{2^m})/2$, let $C_{2^i}$ denote the unique subgroup of order $2^i$. Then
    \[
    [C_{2^m}/e] + [C_{2^m}/C_{2^{m-1}}] = \sum_{i=0}^{m-2} ([C_{2^m}/C_{2^i}] + [C_{2^m}/C_{2^{i+1}}]).
    \]
    Each term $[C_{2^m}/C_{2^i}] + [C_{2^m}/C_{2^{i+1}}]$ is obtained from the basic class
    \[
    [C_4/e] + [C_4/C_2]
    \]
    on the subquotient $C_{2^{i+2}}/C_{2^i} \cong C_4$ by inflation followed by transfer. Since $\underline{\theta}_{8d+1}$ is compatible with inflation and transfer maps, we can compute $\theta_{8d+1}$ of $D_{H,H'}$ via the value of $D_{C_2,\{e\}}$. Therefore, for every cyclic $M\cong C_{2^t} \subset K$, we can compute $\theta_{8d+1}$ by tracking the indicated summand under these inflations and transfers. Let \(H\leq G\), and let \(V\in\operatorname{Irr}_{\mathbb Q}(H)\). Then
    \[
    \operatorname{Tr}_H^G V
    =
    \mathbb Q[G]\otimes_{\mathbb Q[H]}V
    \cong
    \bigoplus_{\substack{
    W\in\operatorname{Irr}_{\mathbb Q}(G)\\
    H\cap\ker(W)=\ker(V)
    }}
    W.
    \]
    Since $\theta_{8d+1}(D_{C_2,{e}})$ is nontrivial only in the summand corresponding to $\rho$, the above description of transfer, together with the behavior of kernels under inflation, shows that, for any $K\leq N_2$, the projection of $\theta_{8d+1}(D_{H,H'})$ onto the summand indexed by $M\cong C_{2^t}\subset K$ is nontrivial if and only if $t\geq 2$ and $\ker(M)$ contains exactly one of $H$ and $H'$. The stated formula then follows by comparing the orders of the corresponding elements.
\end{proof}

\begin{lemma}\label{lem:theta-8d-plus-1-type-two}
Let $2K\subset L\subset T\subset K$ with $T/L\cong C_2\times C_2$. Let $M_1,M_2,M_3$ be the three intermediate subgroups between $L$ and $T$.  For the second-type generator in \Cref{lem:I-8d-plus-1-generators}
\[
B_{T,L} = [K/L]+[K/M_1]+[K/M_2]+[K/M_3] \in I_{8d+1}(N_2/K),
\]
we have 
\[
\theta_{8d+1}(B_{T,L}) = \eta^2u^d \sum_{\substack{\chi:K\to\{\pm1\}\\ L\subseteq\ker(\chi)}}\chi \in RO(K;\mathbb R)/2\{\eta^2u^d\}.
\]
In particular, this value lies entirely in the real-type summand of $\underline{\mathrm{coker}}_2\{8d+2\}(N_2/K)$.
\end{lemma}

\begin{proof}
$\theta_{8d+1}$ is defined by the action of $\underline{A}_{N_2}$ on $\eta u^d$ in $\underline{\pi}_{8d+1} L_{KU_{N_2}/2}S_{N_2}$, so we have 
\[
\theta_{8d+1}(B_{T,L}) = (\operatorname{Tr}_{L}^{K}(1) + \operatorname{Tr}_{M_1}^{K}(1) + \operatorname{Tr}_{M_2}^{K}(1) + \operatorname{Tr}_{M_3}^{K}(1))\cdot \eta u^d.
\]
When $T\cong V=C_2\times C_2$, $B_{T,\{e\}}$ is the image of $X_V\in J(V)$ under the quotient $J(V)\to J(V)/2$, so
\[
\theta_{8d+1}(B_{V,\{e\}}) = X\cdot \eta u^d = \eta^2u^d \rho_{V}.
\]
For a general $L$ and $T$, inflating along $T\to T/L$ and then transferring from $T$ to $K$
gives
\[
\theta_{8d+1}(B_{T,L}) = \eta^2u^d\cdot \operatorname{Tr}_{T}^{K} \operatorname{Inf}_{T/L}^{T} (\rho_{T/L}).
\]
Since $2K\subseteq L$, every character of $K$ trivial on $L$ is a
real one-dimensional character.  Therefore,
\[
\operatorname{Tr}_{T}^{K} \operatorname{Inf}_{T/L}^{T} (\rho_{\mathrm{reg},T/L})= \sum_{\substack{\chi:K\to\{\pm1\}\\ L\subseteq\ker(\chi)}}\chi \in RO(K;\mathbb R)/2.
\]
This proves the claimed formula.
\end{proof}

\begin{proposition}\label{prop:theta-8d-plus-1-determined}
With the notation above, there is an isomorphism of $N_2$-Mackey functors
\[
\underline{\pi}_{8d+1}L_{KU_{N_2}/2}S_{N_2}
\cong
\frac{
\underline A_{N_2}/2
\oplus
\underline{\mathrm{coker}}_2\{8d+2\}
}{
\left\{ r-\theta_{8d+1}(r): r\in I_{8d+1} \right\}
}.
\]
The map $\theta_{8d+1}:I_{8d+1} \to  \underline{\mathrm{coker}}_2\{8d+2\} $ is completely determined by 
\Cref{lem:I-8d-plus-1-generators,lem:theta-8d-plus-1-type-one,lem:theta-8d-plus-1-type-two}.
\end{proposition}
\begin{proof}
    The proof is same as that of \Cref{prop:extension of 2 group KU local sphere-degree-zero}.
\end{proof}

\subsection{Summary of $\underline{\pi}_* L_{KU_{N_p}/p}S_{N_p}$} When $p=2$, we summarize the computation of $\underline{\pi}_* L_{KU_{N_p}/p}S_{N_p}$ as follows.

\begin{proposition}\label{lem:extension of 2 group KU local sphere}
    \[
\underline{\pi}_k L_{KU_{N_2}/2}S_{N_2} \cong \begin{cases}
    \frac{(\underline{A}_{N_2})_2^\wedge \oplus \underline{RO(-;\mathbb{R})}_{N_2}\{\eta\}/2}{\{j-\theta_{N_2}(j):j\in (\underline{J}_{N_2})_2^\wedge\}} & k=0\\    
    \frac{ \underline A_{N_2}/2 \oplus \underline{\mathrm{coker}}_2\{8d+2\} }{\left\{ r-\theta_{8d+1}(r): r\in I_{8d+1} \right\}} & k=8d+1\\
    \underline{\mathrm{coker}}_2\{8d+1\} & k=8d, d\neq 0\\
    \underline{\ker}_2\{8d+2\} & k=8d+2\\
    \underline{\mathrm{coker}}_2\{k+1\} & k=8d+3, 8d+5, 8d+7\\
    0 & \text{otherwise}.
\end{cases}
\]

\end{proposition}

For the sake of self-containment, we list the computation of $\underline{\pi}_* L_{KU_{N_p}/p}S_{N_p}$ in \cite{carawan2023homotopy} for odd prime $p$ as follows:

\begin{lemma}\label{lem:psi-1 for odd p}
    Define $\underline{\mathrm{coker}}_p\{k\} := \mathrm{coker}\ ( \underline{\pi}_k (KU_{N_p})_p^{\wedge} \overset{\psi^g-1}{\longrightarrow} \underline{\pi}_k (KU_{N_p})_p^{\wedge})$ and $\underline{\ker}_p\{k\} := \ker\ (\underline{\pi}_k(KU_{N_p})_p^{\wedge} \overset{\psi^g-1}{\longrightarrow} \underline{\pi}_k (KU_{N_p})_p^{\wedge})$ for an odd prime $p$. Then there are isomorphisms
    \[
    \underline{\pi}_{2d}L_{KU_{N_p}/p}S_{N_p}\cong \underline{\ker}_p\{2d\},\quad \underline{\pi}_{2d-1} L_{KU_{N_p}/p}S_{N_p}\cong \underline{\mathrm{coker}}_p\{2d\}.
    \]
    Moreover, 
    $$\underline{\ker}_p\{2d\}\cong\begin{cases}
        (\underline{R\mathbb{Q}}_p^\wedge)_{N_p}, & d=0\\
        0, & d\neq 0
    \end{cases},$$ 
    \[ \underline{\mathrm{coker}}_p\{2d\}(N_p/K)\cong\begin{cases}
        \bigoplus_{\text{cyclic }T\subset K} \mathbb{Z}_p^\wedge, & d=0\\
        \mathbb{Z}/p^{\nu_p(g^d-1)}  \oplus \bigoplus_{C_{p^k}\subset K, k>0} \mathbb{Z}/p^{k+\nu_p(d)}, & d\neq 0
    \end{cases},\]
    where $\nu_p$ is the $p$-adic valuation. 
\end{lemma}
\begin{proof}
   The computation of $\underline{\ker}_p\{k\}$ is given in \cite[Proposition 6.7]{bonventre2022ku_g}, \cite[Corollary 3.5]{carawan2023homotopy}. The computation of $\underline{\mathrm{coker}}_p\{k\}$ is given in \cite[Proposition 3.7]{carawan2023homotopy}. The assertion about $\underline{\pi}_*L_{KU_{N_p}/p}S_{N_p}$ follows from the fact that $\pi_* KU_{N_p}$ is concentrated in even degrees.
\end{proof}

%%-------------------------------------------------------------------------------------------------------------------------------------------------------------------------------------
\section{The homotopy Mackey functor of $L_{KU_G}S_G$}\label{Sec:KU_G local sphere}
In this section, let $G$ be a finite abelian group with Sylow $p$-subgroup $N_p$, and let $N$ denote the product of the Sylow $q$-subgroups of $G$ for $q\neq p$. We compute $\underline{\pi}_*L_{KU_G}S_G$ in \Cref{thm:integer graded homotopy Mackey functor of ku local sphere}.

In order to compute $\underline{\pi}_0L_{KU_G/p}S_G$, we need the following lemma. 
\begin{lemma}\label{lem:tensor of A/J}
    For any $K\subset N_p$ and $H\subset N$, 
    \[
    A/J(K)\otimes A/J(H)\cong A/J(K\oplus H).
    \]
    Then there is an isomorphism of $G$-Mackey functors $\underline{A/J}_{N_p}\otimes \underline{A/J}_N\cong \underline{A/J}_G$. 
\end{lemma}
\begin{proof}
    We first show that the map $A(K)\otimes A(H)\to A(K\oplus H)$ is an isomorphism. For any $X\in A(K)$ and $Y\in A(H)$, there is a natural action of $K\oplus H$ on $X\times Y$; thus, $(X,Y)\mapsto X\times Y$ induces a ring homomorphism 
    \[
    f:A(K)\otimes A(H)\to A(K\oplus H). 
    \]
    Conversely, as observed in the proof of \Cref{lem:tensor product}, any $Z\in A(K\oplus H)$ can be expressed uniquely as
    \[ Z\cong \bigsqcup_{i=1}^r X_i\times Y_i,\]
    where $X_i$, $Y_i$ are $K$-orbits and $H$-orbits, respectively. Therefore, there is a ring homomorphism
    \[
    g:A(K\oplus H)\to A(K)\otimes A(H),\quad Z\mapsto \sum_{i=1}^r (X_i,Y_i),
    \]
    which is the inverse of $f$. Since any subgroup $S \subset K \oplus H$ satisfies $S \cong (S\cap K) \oplus (S\cap H)$, $f$ commutes with $\operatorname{Res}_{S}^{K \oplus H}$ and $\operatorname{Tr}_{S}^{K \oplus H}$, $f$ induces an isomorphism of $G$-Green functors $\underline{A}_{N_p} \otimes \underline{A}_N \to \underline{A}_G$.
    
    Consider the following commutative diagram with exact rows
    \[
    \begin{tikzcd}
    0 \ar[r] & \underline{\ker}(\mathcal{R}_{N_p}\otimes \mathcal{R}_N) \ar[r] \ar[d,dashed] & \underline{A}_{N_p}\otimes\underline{A}_N \ar[r,"\mathcal{R}_{N_p}\otimes \mathcal{R}_N"]\ar[d,"\cong"] &   \underline{RU}_{N_p}\otimes \underline{RU}_N \ar[d,"\cong"]\\
    0 \ar[r] & \underline{J}_G \ar[r]   & \underline{A}_G \ar[r,"\mathcal{R}_G"] &\underline{RU}_G
    \end{tikzcd},
    \]
    there exists an isomorphism $\underline{\ker}(\mathcal{R}_{N_p}\otimes \mathcal{R}_N)\cong \underline{J}_G$ indicated by the dashed arrow. Thus \[\underline{A/J}_{N_p}\otimes \underline{A/J}_N\cong \mathrm{Im}(\mathcal{R}_{N_p}\otimes \mathcal{R}_N)\cong \mathrm{Im}\mathcal{R}_G\cong \underline{A/J}_G.\]
    Here the first isomorphism follows from the fact that $RU(K)$ is a free abelian group for every subgroup $K\subset G$, and the second follows from the Five Lemma.
\end{proof}

Combining the above calculations, we give $\underline{\pi}_* L_{KU_G/p}S_G$ as follows:
\begin{proposition}\label{thm: Z graded homotopy of KU_G/p local sphere}
    Let $G = N_p \oplus N$, where $N_p$ is the Sylow $p$-subgroup. For $p=2$, 
    \[\underline{\pi}_n(L_{KU_G/2}S_G) \cong \begin{cases}
        \frac{(\underline{A}_{N_2})_2^\wedge \oplus \underline{RO(-;\mathbb{R})}_{N_2}\{\eta\}/2}{\{j-\theta_{N_2}(j):j\in (\underline{J}_{N_2})_2^\wedge\}} \otimes \underline{A/J}_N, & n=0\\
        \underline{\mathrm{coker}}_2\{0\}\otimes \underline{A/J}_N  & n=-1\\
        \underline{\mathrm{coker}}_2\{8d+1\}\otimes \underline{A/J}_N  & n=8d, d\neq 0\\
        \underline{\pi}_{8d+1} L_{KU_{N_2}/2}S_{N_2}\otimes \underline{A/J}_N & n=8d+1\\
        \underline{\ker}_2\{8d+2\}\otimes \underline{A/J}_N & n=8d+2\\
        \underline{\mathrm{coker}}_2\{n+1\}\otimes \underline{A/J}_N & n=8d+3, 8d+5,\\&\qquad 8d+7,n\neq -1\\
        0 & \text{otherwise}.
    \end{cases}\]
    For odd prime $p$, 
    \[
    \underline{\pi}_n(L_{KU_G/p}S_G) \cong \begin{cases}
        (\underline{A/J}_G)_p^\wedge, & n=0\\
        \underline{\mathrm{coker}}_p\{0\}\otimes \underline{A/J}_N  & n=-1\\
        \underline{\mathrm{coker}}_p\{2d\}\otimes \underline{A/J}_N  & n=2d-1, d\neq 0\\
        0 & \text{otherwise}
    \end{cases}.
    \]
    Here $\underline{\ker}_p\{k\}$ and $\underline{\mathrm{coker}}_p\{k\}$ are listed in \Cref{lem:psi-1,lem:psi-1 for odd p}, $\theta_{N_2}$ is given in \Cref{prop:extension of 2 group KU local sphere-degree-zero}, and $\underline{\pi}_{8d+1} L_{KU_{N_2}/2}S_{N_2}$ is listed in \Cref{prop:theta-8d-plus-1-determined}. In particular, $\underline{\pi}_n(L_{KU_G/p}S)$ is torsion-free when $n = -1$. For $n \neq 0, -1$, the Mackey functors $\underline{\pi}_n(L_{KU_G/p}S_G)$ are torsion. 
\end{proposition}
\begin{proof}
    The computation follows from \Cref{lem:Z graded homotopy Mackey functor of ECyc inf} and \Cref{lem:extension of 2 group KU local sphere}. Note that by \Cref{lem:tensor of A/J},
    \[
    (\underline{R\mathbb{Q}}_{N_p}\otimes \underline{A/J}_N)_p^\wedge\cong (\underline{A/J}_{N_p}\otimes \underline{A/J}_N)_p^\wedge\cong (\underline{A/J}_G)_p^\wedge.
    \]
    Then $\underline{\pi}_0(L_{KU_G/p}S_G) \cong (\underline{A/J}_G)_p^\wedge$ when $p$ is odd.
\end{proof}

    Finally, we compute $\underline{\pi}_*L_{KU_G}S_G$ via the Arithmetic fracture square
    \[\begin{tikzcd}[ampersand replacement=\&]
        L_{KU_G}S_G  \ar[r] \ar[d] \& \prod_{p} L_{KU_G/p} S_G \ar[d] \\
        L_{KU_G\otimes\mathbb{Q}} S_G \ar[r] \& (\prod_{p} L_{KU_G/p}S_G)_{\mathbb{Q}} .
    \end{tikzcd}\]
    
\begin{theorem}\label{thm:integer graded homotopy Mackey functor of ku local sphere}
	Let $G$ be a finite abelian group. For any prime $p$, let $N_p$ be the Sylow $p$-subgroup of $G$, and let $G/N_p$ denote the product of the Sylow $q$-subgroups of $G$ for $q\neq p$.
    \[
	\underline{\pi}_kL_{KU_G}S_G\cong \begin{cases}
		\underline{A/J}_{G/N_2} \otimes  \frac{\underline{A}_{N_2} \oplus \underline{RO(-;\mathbb{R})}_{N_2}\{\eta\}/2}{\{j-\theta_{N_2}(j):j\in \underline{J}_{N_2}\}} & \quad k=0\\
		0 & \quad k=-1\\
		\mathbb{Q}/\mathbb{Z}\otimes (\prod_p \underline{\operatorname{coker}}_p\{0\}\otimes \underline{A/J}_{G/N_p}) & \quad k=-2\\
		\prod_p \underline{\pi}_kL_{KU_G/p}S_G & \quad \text{otherwise}
	\end{cases}
	\]
    Here $\theta_{N_2}:\underline{J}_{N_2} \to \underline{RO(-;\mathbb{R})}_{N_2}\{\eta\}/2$ is induced by the Hurewicz map $\underline{A}_{N_2}\cong \underline{\pi}_0 S_{N_2}\to \underline{\pi}_0 L_{KU_{N_2}/2}S_{N_2}$. The computation of $\theta_{N_2}$ is carried out in \Cref{prop:extension of 2 group KU local sphere-degree-zero}.
\end{theorem}
\begin{proof}
    For $KU_G\otimes \mathbb{Q}$-localization, it follows from \cite[Lemma 9.1]{bonventre2022ku_g} that $L_{KU_G \otimes \mathbb{Q}} S_G \simeq H(\mathbb{Q} \otimes \underline{A/J})$, which is a $G$-equivariant Eilenberg-MacLane spectrum. So 
    \[
    \underline{\pi}_* L_{KU_G \otimes \mathbb{Q}} S_G \cong \mathbb{Q} \otimes \underline{A/J}   
    \]
    concentrated in degree $0$. Furthermore, since $\underline{\pi}_k \prod_{p} L_{KU_G/p} S_G$ are torsion groups for $k \neq 0, -1$, $(\prod_{p} L_{KU_G/p} S_G)_{\mathbb{Q}}$ has non-trivial homotopy groups only in degrees $0$ and $-1$. Consequently, the long exact sequence on homotopy groups induced by the arithmetic fracture square takes the following form:
    \[
    \begin{aligned}
        0\to & \underline{\pi}_0 L_{KU_G}S_G \to \underline{\pi}_0(\prod_{p} L_{KU_G/p} S_G)\bigoplus (\mathbb{Q}\otimes \underline{A/J}) \overset{f}{\longrightarrow} \underline{\pi}_0(\prod_{p} L_{KU_G/p} S_G)\otimes \mathbb{Q} \\
        & \to \underline{\pi}_{-1}L_{KU_G}S_G \to \underline{\pi}_{-1}(\prod_{p} L_{KU_G/p} S_G) \overset{g}{\longrightarrow} \underline{\pi}_{-1}(\prod_{p} L_{KU_G/p} S_G)\otimes \mathbb{Q} \\
        & \to \underline{\pi}_{-2}L_{KU_G}S_G \to 0,
    \end{aligned}
    \] 
    and for $k\neq 0,-1,-2$, 
    \[
    0\to \underline{\pi}_{k}L_{KU_G}S_G \to \underline{\pi}_{k}(\prod_{p} L_{KU_G/p} S_G)\to 0.
    \]
    Then for $k\neq 0,-1,-2$, $\underline{\pi}_kL_{KU_G}S_G\cong\prod_p \underline{\pi}_kL_{KU_G/p}S_G$. 

    When $k=0$, $\underline{\pi}_0 L_{KU_G}S_G\cong \ker f$. Let 
    $$M=\frac{\underline{A}_{N_2} \oplus \underline{RO(-;\mathbb{R})}_{N_2}\{\eta\}/2}{\{j-\theta_{N_2}(j):j\in \underline{J}_{N_2}\}},$$
    we have $M_2^\wedge \cong \underline{\pi}_0 L_{KU_{N_2}/2}S_{N_2}$. Note that all elements of $\underline{RO(-;\mathbb{R})}_{N_2}\{\eta\}/2$ are $2$-torsion, after $p$-completion for odd prime $p$ or rationalization, $M_p^\wedge \cong (\underline{A/J}_{N_2})_p^\wedge$, and $M\otimes \mathbb{Q}\cong \underline{A/J}_{N_2}\otimes \mathbb{Q}$. Therefore,
    \[
    \underline{A/J}_p^\wedge \cong \underline{A/J}_{G/N_2} \otimes M_p^\wedge, \quad \underline{A/J}\otimes \mathbb{Q} \cong \underline{A/J}_{G/N_2} \otimes M\otimes \mathbb{Q}.
    \]
    By \Cref{thm: Z graded homotopy of KU_G/p local sphere}, $f$ is given by 
    \[
    f : (\prod_p M_p^\wedge \bigoplus M\otimes \mathbb{Q})\otimes \underline{A/J}_{G/N_2} \to (\prod_p M_p^\wedge )\otimes \mathbb{Q} \otimes \underline{A/J}_{G/N_2}.
    \]
    This is exactly the arithmetic pullback of $ M\otimes \underline{A/J}_{G/N_2}$, then
    \[\underline{\pi}_0 L_{KU_G}S_G\cong \ker f \cong  M\otimes \underline{A/J}_{G/N_2}.\]

    When $k=-1$, it follows from \cite[Lemma 4.14]{carawan2023homotopy} that $\mathrm{coker} \ f=0$, so $\underline{\pi}_{-1} L_{KU_G}S_G\cong \ker g$. Here 
    \[
    g : \underline{\pi}_{-1}(\prod_{p} L_{KU_G/p} S_G) \to \underline{\pi}_{-1}(\prod_{p} L_{KU_G/p} S_G)\otimes \mathbb{Q}
    \]
    is an inclusion induced by $\mathbb{Z}\to \mathbb{Q}$ since $\underline{\pi}_{-1}(\prod_{p} L_{KU_G/p} S_G)$ is torsion free, thus $\underline{\pi}_{-1} L_{KU_G}S_G\cong \ker g=0$. 

    When $k=-2$, $\underline{\pi}_{-2} L_{KU_G}S_G\cong \mathrm{coker}\ g$. Therefore, 
    \[\underline{\pi}_{-2} L_{KU_G}S_G\cong \underline{\pi}_{-1}(\prod_{p} L_{KU_G/p} S_G)\otimes \mathbb{Q}/\mathbb{Z}\cong \mathbb{Q}/\mathbb{Z} \otimes (\prod_p \underline{\operatorname{coker}}_p\{0\}\otimes \underline{A/J}_{G/N_p}).\]
\end{proof}

\begin{remark}[A remark for $\underline{\pi}_0$]
    If $N_2$ is cyclic, then $J(N_2)=\{0\}$, and the extension problem for $k=0$ in \Cref{prop:extension of 2 group KU local sphere-degree-zero} is trivial. In this case, 
    \[\underline{\pi}_0 L_{KU_G}S_G\cong \underline{A/J}\bigoplus \underline{A/J}_{G/N_2}\otimes \underline{RO(-;\mathbb{R})}/2.\] 
\end{remark}

\begin{remark}\label{rmk: KU_G local for nilpotent G}
    It follows from the proof that \Cref{thm:integer graded homotopy Mackey functor of ku local sphere} applies to every finite nilpotent group $G$ whose Sylow $2$-subgroup $N_2$ is abelian. When $N_2$ is non-abelian, we can compute $\mathrm{coker}\{k\}$ and $\ker\{k\}$ similarly, and then the classification of the Brauer relations in \cite[Theorem 5.3]{bartel2015brauer} implies that it remains to resolve the extension problem for $N_2\cong D_{2^n}$, in a manner analogous to \Cref{lem:case N_2=V}.
\end{remark}

%%-------------------------------------------------------------------------------------------------------------------------------------------------------------------------------------

\section{Relation with equivariant Morava K-theory}\label{Sec:equivariant chromatic}

In this section, let $G$ be a finite group. For a subgroup $H\subset G$, let $N_G H$ denote its normalizer and $W_G H=N_G H/H$ its Weyl group. Let $\mathrm{Cyc}$ denote the family of cyclic subgroups of $G$. For a fixed prime $p$, we study the role of $KU_G/p$ in equivariant chromatic homotopy theory, which yields a splitting of the $KU_G/p$-localization functor in \Cref{thm:KUG local spectrum second case}. We then study the $RO(G)$-graded homotopy Mackey functors of $L_{KU_G/p}S_G$ for finite nilpotent groups $G$.

\begin{lemma}\label{lem:TH local}
    For $H\subset G$, let $T(H)=G/H_+\wedge S[\mathcal{F}_{H\not\subset}^{-1}]\in \mathrm{Sp}^G$. Then for any $X\in \mathrm{Sp}$, we have 
    \[
      L_{T(H)} \operatorname{Inf}_e^G X \simeq \operatorname{CoInd}_{N_GH}^G\bigl( \operatorname{Inf}_{W_G H}^{N_G H} F(EW_G H_+, X)[\mathcal{F}_{H\not\subset}^{-1}] \bigr).
    \]
\end{lemma}
\begin{proof}
    Let $R=\operatorname{CoInd}_{N_GH}^G\bigl( \operatorname{Inf}_{W_G H}^{N_G H} F(EW_G H_+, X)[\mathcal{F}_{H\not\subset}^{-1}] \bigr)$. For any $T(H)$ acyclic $G$-spectrum $M$, 
    $\Phi^H(M\wedge T(H))\simeq W_G H_+\wedge \Phi^H M\simeq \ast$, thus the underlying spectrum of $\Phi^H M$ is contractible. So $EW_G H_+\wedge \Phi^H M$ is equivariantly contractible as a $W_G H$-spectrum.
    By \Cref{prop:pushout type}, there are isomorphisms
    \[
    \begin{aligned}
        [M,R]^G
        \cong & [\operatorname{Res}_{N_GH}^G M, \operatorname{Inf}_{W_G H}^{N_G H} F(EW_G H_+, X)[\mathcal{F}_{H\not\subset}^{-1}]]^{N_GH}\\
        \cong & [\Phi^H M, F(EW_GH_+, X)]^{W_G H}\cong 0.
    \end{aligned}
    \]
    Therefore, $R$ is $T(H)$-local. 

    The map 
    $i:X\to R$
    induced by $EW_G H_+\to S^0$ is a $T(H)$-equivalence. Indeed, if $N$ is not conjugate to $H$, then $\Phi^N(i \wedge T(H))$ is a map between trivial spectra. For any subgroup $K$ conjugate to $H$, it suffices to compute $\Phi^H(i\wedge T(H))$. This is the map $X\wedge W_GH_+\to F(EW_GH_+, X)\wedge W_G H_+$ induced by $EW_GH_+\to S^0$, and it is an equivalence of underlying non-equivariant spectra.
\end{proof}

\begin{lemma}\label{lem:morava k local}
    Let $G$ be a finite group, and let $H\subset G$ be a subgroup. For any $E\in \mathrm{Sp}$ and $X\in \mathrm{Sp}^G$, we have
    \[
    \begin{aligned}
        L_{T(H)\wedge E}X  & \simeq \operatorname{CoInd}_{N_GH}^G\bigl( \operatorname{Inf}_{W_G H}^{N_G H} F(EW_G H_+, L_{E}\Phi^H X)[\mathcal{F}_{H\not\subset}^{-1}] \bigr)\\
        & \simeq L_{T(H)}L_{E}\Phi^H X.
    \end{aligned}
    \]
\end{lemma}
\begin{proof}
    Let $R=\operatorname{CoInd}_{N_GH}^G\bigl( \operatorname{Inf}_{W_G H}^{N_G H} F(EW_G H_+, L_{E}\Phi^H X)[\mathcal{F}_{H\not\subset}^{-1}] \bigr)$, we need to show that $R$ is $T(H)\wedge E$-local, and there exists a $T(H)\wedge E$-equivalence $X\to R$. 

    By \Cref{lem:TH local}, $R\simeq L_{T(H)}L_{E}\Phi^H X$ is $T(H)$-local. For any $T(H)\wedge E$-acyclic $G$-spectrum $M$, 
    \[
    \begin{aligned}
        [M, R]^G &\cong [\operatorname{Res}_{N_G H}^G M, \operatorname{Inf}_{W_G H}^{N_G H} F(EW_G H_+, L_{E}\Phi^H X)[\mathcal{F}_{H\not\subset}^{-1}]]^{N_GH}\\
        & \cong [\Phi^H M, F(EW_G H_+, L_{E}\Phi^H X)]^{W_G H}\\
        & \cong [EW_G H_+ \wedge \Phi^H M, L_{E}\Phi^H X]^{W_G H}.
    \end{aligned}
    \]
    The second isomorphism follows from \Cref{prop:pushout type}. Since $M\wedge T(H)\wedge E\simeq \ast$, the spectrum $\Phi^H M$ is $W_G H_+\wedge E$-acyclic. Hence $EW_G H+\wedge \Phi^H M$ is $E$-acyclic. It follows that $[M,R]^G\cong 0$, so $R$ is $E$-local.

    There is a map $i: X\to R$ induced by the $E$-localization $\Phi^H X\to L_E \Phi^H X$ via the isomorphism 
    \[
    [X,R]^G\cong [\Phi^H X, F(EW_G H_+, L_{E}\Phi^H X)]^{W_G H}. 
    \]
    For any $N\subset G$, $\Phi^N (T(H)\wedge E) \simeq \ast$ if $N$ is not conjugate to $H$. When $N\sim^G H$, $\Phi^N(i\wedge T(H)\wedge E)$ is equivalent to the map
    \[
    \Phi^H X\wedge W_GH_+ \wedge E \longrightarrow L_{E}\Phi^H X \wedge W_G H_+ \wedge E
    \]
    induced by the $E$-localization $\Phi^H X\to L_{E} \Phi^H X$, which is an equivalence. Therefore, $\Phi^N(i\wedge T(H)\wedge E)$ is an equivalence for all $N\subset G$. Hence $i$ is a $T(H)\wedge E$-equivalence.
\end{proof}

Given a subgroup $H\subset G$, a positive integer $n$ and a prime $p$, let 
    \[
    K(H,n):=G/H_+\wedge K(n) [\mathcal{F}_{H\not\subset}^{-1}]\in \mathrm{Sp}_{(p)}^G.
    \]
Then \Cref{lem:morava k local} implies that for every $X\in \mathrm{Sp}^G$, $L_{K(H,n)} X \simeq L_{T(H)}L_{K(n)}\Phi^HX$. When $X=S_G$, we have $\Phi^H S_G\simeq S$ for all $H\subset G$. For any odd prime $p$, let $g=(\zeta_{p-1},p+1)$ be a topological generator of $\mathbb{Z}_p^{\times}$ and $B=KU_p^\wedge$. Applying the functor $L_{T(H)}(\operatorname{Inf}_e^G-)$ to the sequence $L_{K(1)}S\to B\overset{\psi^g-1}{\longrightarrow} B$, we have a fiber sequence
    \[
    L_{K(H,1)}S_G\rightarrow L_{T(H)}B\overset{\psi^g-1}{\longrightarrow}L_{T(H)}B.
    \]
When $p=2$, the same statement holds with $g$ a generator of $(\mathbb{Z}_2^\wedge)^{\times}/{\pm 1}$ and $B=KO_2^\wedge$.

\begin{lemma}\label{lem: double coset formula}
    Let $H,K\subset G$ be subgroups, and let $X$ be a $K$-spectrum. Then there is an equivalence of $H$-spectra
    \[
    \operatorname{Res}_H^G \operatorname{CoInd}_K^G X \simeq \prod_{[g]\in H\backslash G/K} \operatorname{CoInd}_{H\cap {}^g K}^H \operatorname{Res}_{H\cap {}^g K}^{{}^g K} {}^g X.
    \]
    Here ${}^g K=gKg^{-1}$ and ${}^g X=c_g^* X$, where $c_g: {}^g K\to K$ is the conjugation map.
\end{lemma}
\begin{proof}
    We can decompose the $(K,H)$-biset $G$ as a disjoint union 
    \[
    G= \bigsqcup_{[g]\in H\backslash G/K} H\cdot g \cdot K.
    \]
    Then 
    \[
    \begin{aligned}
        \operatorname{Res}_H^G \operatorname{CoInd}_K^G X &\simeq \operatorname{Res}_H^G F_K(G_+,X)\\
        &\simeq \prod_{[g]\in H\backslash G/K} F_K((H\cdot g \cdot K)_+,X)\\
        &\simeq \prod_{[g]\in H\backslash G/K} F_{H\cap {}^g K}(H_+, \operatorname{Res}_{H\cap {}^g K}^{{}^g K} {}^g X)\\
        &\simeq \prod_{[g]\in H\backslash G/K} \operatorname{CoInd}_{H\cap {}^g K}^H \operatorname{Res}_{H\cap {}^g K}^{{}^g K} {}^g X.
    \end{aligned}
    \]  
\end{proof}

\begin{lemma}\label{lem:bousfield class via morava k}
    $KU_G/p$ is Bousfield equivalent to $\bigvee_{(H) \in Cyc/G, p\nmid |H|} K(H,1)$, where $H$ ranges over the conjugacy classes of cyclic subgroups of $G$ such that $p\nmid |H|$.
\end{lemma}
\begin{proof}
    This follows from the fact that 
    \[ 
    \Phi^N K(H,n)\simeq\begin{cases}
        (G/H_+)^N\wedge K(n) & N\sim^G H,\\
        \ast & \text{otherwise.}
    \end{cases}
    \]
\end{proof}

    Note that $K(H,n)$-local objects are automatically $p$-complete, which implies that $L_{T(H)}L_{K(n)}\Phi^H X$ is $p$-complete.

\begin{lemma}\label{lem:step by step}
    Let $I$ be a set of subgroups of $G$, and let $H$ be a subgroup of $G$ such that $H$ is not subconjugate to $T$ for any $T\in I$. Then for any $X\in \mathrm{Sp}^G$, there is a pullback square
    \[
    \begin{tikzcd}
        L_{\bigvee_{T\in I\cup\{H\}}K(T,n)}X \arrow[r]\arrow[d] & L_{K(H,n)} X\arrow[d]\\
        L_{\bigvee_{T\in I}K(T,n)}X \arrow[r] & L_{K(H,n)}L_{\bigvee_{T\in I}K(T,n)}X.
    \end{tikzcd}
    \]
\end{lemma}
\begin{proof}
    By \Cref{prop:arithmetic square}, it suffices to show that 
    \[
    (\bigvee_{T\in I}K(T,n)) \wedge L_{K(H,n)} X\simeq \ast. 
    \]
    Indeed, $\Phi^L K(T,n)\not\simeq \ast$ if and only if $L$ is conjugate to $T$. 
    
    For $L\subset G$, consider the $L$-geometric fixed point of $L_{K(H,n)}X$. If $L$ is not subconjugate to $N_G H$, then for every $g\in N_G L\backslash G/N_G H$, $\operatorname{Res}_{N_GH\cap^g N_G L}^{N_G L}\widetilde{E\mathcal F_{L\not\subset}}$ is trivial, and
    \[
        (\operatorname{Res}_{N_G L}^{G} L_{K(H,n)}X)\wedge \widetilde{E\mathcal F_{L\not\subset}}  \simeq \ast.
    \]
    If $L\subset N_G H$, then $\Phi^L L_{K(H,n)}X$ is nontrivial only if $H$ is subconjugate to $L$. Since $H$ is not subconjugate to $T$ for all $T\in I$, we have 
    \[
    \Phi^L((\bigvee_{T\in I}K(T,n)) \wedge L_{K(H,n)} X)\simeq \ast
    \]
    for all $L\subset G$, i.e., $(\bigvee_{T\in I}K(T,n)) \wedge L_{K(H,n)} X\simeq \ast$.
\end{proof}

\begin{theorem}\label{thm:KUG local spectrum second case}
    Let $G$ be a finite group, and let $Cyc$ be the family of cyclic subgroups of $G$. For any prime $p$ and any $G$-spectrum $X$, there is an equivalence of $G$-spectra
    \[
    L_{KU_G/p}X \simeq L_{\bigvee_{(H)\in Cyc/G, \, p\nmid |H|}K(H,1)}X \simeq \bigvee_{(H)\in Cyc/G, \, p\nmid |H|} L_{K(H,1)}X.
    \]
    Here $H$ ranges over the conjugacy classes of cyclic subgroups of $G$ such that $p\nmid |H|$.
\end{theorem}
\begin{proof}
    Let $I$ be a set of representatives for the conjugacy classes of cyclic subgroups $H \subset G$ such that $p\nmid |H|$. We have $L_{KU_G/p}X\simeq L_{\bigvee_{H\in I}K(H,1)}X$ since $\langle KU_G/p\rangle=\langle \bigvee_{H\in I}K(H,1)\rangle$. 

    It follows from \Cref{lem:morava k local} that for any subgroups $H_1, H_2\in I$,
    \[
        L_{K(H_1,1)}L_{K(H_2,1)}X = L_{T(H_1)}L_{K(1)}\Phi^{H_1}(L_{T(H_2)}L_{K(1)}\Phi^{H_2}X).
    \]
    Here $p\nmid |H_1|$, by \Cref{prop:q local geometric fixed points} we have 
    \[
    \begin{aligned}
        \Phi^{H_1}(L_{T(H_2)}L_{K(1)}\Phi^{H_2}X) & \simeq \Phi^{H_1}(L_{T(H_2)/p}L_{K(1)}\Phi^{H_2}X)\\
        &\simeq L_{\Phi^{H_1}T(H_2)/p}(\Phi^{H_1}L_{K(1)}\Phi^{H_2}X) \\
        &\simeq \begin{cases}
            L_\ast(\Phi^{H_1}L_{K(1)}\Phi^{H_2}X)\simeq \ast & H_1\not\sim^G H_2\\
            L_{K(H_1,1)}X & H_1\sim^G H_2.
        \end{cases}
    \end{aligned}
    \]
    Therefore, $L_{K(H_1,1)}L_{K(H_2,1)}X\simeq \ast$ whenever $H_1$ is not conjugate to $H_2$ in G and $p\nmid |H_1|$, and the pullback square in \Cref{lem:step by step} implies that  
    \[
    L_{K(H_1,1)\vee K(H_2,1)}X\simeq L_{K(H_1,1)}X \vee L_{K(H_2,1)}X.
    \]

    We can order the elements of $I$ by the order of the corresponding subgroups. For subgroups of the same order, we assign an arbitrary order, since they cannot subconjugate to one another. Based on this order, $I$ is totally ordered. We write
    \[I=\{H_1,H_2,\ldots,H_n\},\]
    where $|H_i|\leq |H_{i+1}|$. Let $I_1=\{H_1\}$ and $I_k=I_{k-1} \cup \{H_k\}$. 
    
    We study $L_{\bigvee_{H\in I}K(H,1)}S_G$ by induction on $I_k$, using the pullback diagrams in \Cref{lem:step by step}. Assume that for any $k<n$, $L_{\bigvee_{H\in I_k}K(H,1)}X \simeq \bigvee_{H\in I_k} L_{K(H,1)}X$. Then
    \[ 
    L_{K(H_n,1)}L_{\bigvee_{H\in I_{n-1}}K(H,1)}X\simeq \bigvee_{H\in I_{n-1}}L_{K(H_n,1)}L_{K(H,1)}X\simeq \ast,
    \]
    thus $L_{\bigvee_{H\in I_n}K(H,1)}X \simeq \bigvee_{H\in I_n} L_{K(H,1)}X$.
\end{proof}

When $G$ is nilpotent, this theorem reduces the computation of $\underline{\pi}_V L_{KU_G/p}S_G$ for $V\in RO(G)$ to the case of $p$-groups. Let $G=N_p\times N$ be a finite nilpotent group with Sylow $p$-subgroup $N_p$. For $V\in RO(G)$, choose real $G$-representations $V_+$ and $V_-$ such that $V=V_+-V_-$, and define 
\[
\epsilon_V=\det(V_+)\det(V_-)^{-1} \colon G\to \{\pm 1\}.
\]
For each subgroup $H\subset G$, let $\epsilon_{V,H}=\epsilon_V|_H$ denote the restriction of $\epsilon_V$ to $H$. If $H$ is conjugate to $Q$, then $\epsilon_{V,H}\equiv 1$ if and only if $\epsilon_{V,Q}\equiv 1$. With this notation, we have the following lemma.

\begin{lemma}\label{lem: stable splitting of Thom speces}
For every $V\in RO(G)$, there is a weak equivalence of $p$-local spectra
\[
    (EG_+\wedge_G S^V)_{(p)} \simeq \begin{cases}
        ((EN_p)_+\wedge_{N_p} S^{\operatorname{Res}_{N_p}^G V})_{(p)} & \text{if } \epsilon_{V,N}\equiv 1,\\
        \ast & \text{otherwise.} 
    \end{cases}
\]
\end{lemma}
\begin{proof}
    Let $d=\dim V=\dim V_+-\dim V_-$. There is a homological homotopy-orbit spectral sequence \cite[Proposition B.7]{campbell2024hilbert}
    \[
    E_{s,t}^2=H_s(G;\tilde{H}_t(S^V;\mathbb Z_{(p)}))\Longrightarrow \tilde{H}_{s+t}(EG_+\wedge_G S^V;\mathbb Z_{(p)}).
    \]
    As a $G$-module,
    \[
    \tilde{H}_{t}(S^V;\mathbb Z_{(p)})\cong \begin{cases}
        \mathbb Z_{(p)}\{\epsilon_V\} & t=d,\\
        0 & t\neq d,
    \end{cases}
    \]
    where $\mathbb Z_{(p)}\{\epsilon_V\}$ is the $G$-module $\mathbb Z_{(p)}$ with the $G$-action defined by
    $$g\cdot 1 = \epsilon_V(g) \quad \forall g\in G.$$
    The spectral sequence therefore collapses at the $E^2$-page, yielding the isomorphisms
    \[
    \tilde{H}_{s+d}(EG_+\wedge_G S^V;\mathbb Z_{(p)})\cong H_s(G;\mathbb Z_{(p)}\{\epsilon_V\}).
    \]
    Since $G=N_p\times N$, the extension 
    \[
    \{1\}\to N \to G \to N_p \to \{1\}
    \] 
    gives a Lyndon-Hochschild-Serre spectral sequence
    \[
    E^2_{s,t}=H_s(N_p;H_t(N;\mathbb Z_{(p)}\{\epsilon_V\}))\Longrightarrow H_{s+t}(G;\mathbb Z_{(p)}\{\epsilon_V\}).
    \]
    Since $p\nmid |N|$,the augmentation $\mathbb Z_{(p)}[N]\to \mathbb Z_{(p)}$ admits a splitting as a homomorphism of $\mathbb Z_{(p)}[N]$-modules, and the trivial $\mathbb Z_{(p)}[N]$-module $\mathbb{Z}_{(p)}$ is a projective $\mathbb Z_{(p)}[N]$-module. It follows that , 
    \[
    H_t(N;\mathbb Z_{(p)}\{\epsilon_V\})\cong \operatorname{Tor}_t^{\mathbb Z_{(p)}[N]}(Z_{(p)},\mathbb Z_{(p)}\{\epsilon_V\})\cong\begin{cases}
        \mathbb Z_{(p)}\{\epsilon_{V}\}/N & t=0\\
        0 & t>0
    \end{cases},
    \]
    where $\mathbb Z_{(p)}\{\epsilon_V\}/N$ is the coinvariant of $\mathbb Z_{(p)}\{\epsilon_V\}$ under the $N$-action. 
    Thus the spectral sequence is concentrated in the zero-th row, and collapses at $E^2$-page. Therefore,
    \[
    H_s(G;\mathbb Z_{(p)}\{\epsilon_V\})\cong H_s(N_p;\mathbb Z_{(p)}\{\epsilon_V\}/N).
    \]

    If $\epsilon_{V,N}\equiv 1$, then $N$ acts trivially on $\mathbb Z_{(p)}\{\epsilon_V\}$, and 
    \[
    H_s(N_p;\mathbb Z_{(p)}\{\epsilon_V\})\cong H_s(G;\mathbb Z_{(p)}\{\epsilon_V\}),
    \]
    It follows from the homotopy-orbit spectral sequence that the inclusion $N_p\to G$ induces a map $$\alpha: (EN_p)_+\wedge_{N_p} S^{\operatorname{Res}_{N_p}^G V}\to EG_+\wedge_G S^V,$$ which is a $\mathbb Z_{(p)}$-homology equivalence. After $p$-localization, it induces a weak equivalence $\alpha_{(p)}$ of $p$-local spectra.
    
    If $\epsilon_{V,N}\not\equiv 1$, then there exists $n\in N$ such that $\epsilon_V(n)=-1$. Note that a nontrivial homomorphism $\epsilon_V: N \to \{\pm 1\}\cong C_2$ forces $p\neq 2$. Then $2$ is invertible in $\mathbb Z_{(p)}$, and $\mathbb Z_{(p)}\{\epsilon_V\}/N=0$. Therefore, $$H_\ast(G;\mathbb Z_{(p)}\{\epsilon_V\})=0,$$
    which implies that $(EG_+\wedge_G S^V)_{(p)}$ is contractible.
\end{proof}
    
    Since $p\nmid |N|$, there is an idempotent decomposition of the Burnside ring $A(N)_{(p)}$ \cite[Appendix A]{greenlees1995generalized} 
    \[
    \phi_N\colon A(N)_{(p)}\cong \prod_{(K)\in \operatorname{Sub}(N)/N} \mathbb Z_{(p)},
    \]
    Let $e_{(K)}^N\in A(N)_{(p)}$ denote the idempotent corresponding to the factor $\mathbb Z_{(p)}$ indexed by the conjugacy class $(K)$ of $K\subset N$. Then $e_{(K)}^N$ is characterized by 
    \[
    \phi_N^H(e_{(K)}^N)=\begin{cases}
        1 & \text{if } K\sim^N H,\\
        0 & \text{otherwise.}
    \end{cases}
    \]
    For $H\subset N$ and $V\in RO(N)$, let 
    \[
    e_{(H)}^N S_N=\mathrm{hocolim}(S_{N,(p)}\overset{e_{(H)}^N}{\longrightarrow} S_{N,(p)} \overset{e_{(H)}^N}{\longrightarrow} S_{N,(p)} \longrightarrow\cdots)
    \] 
    be the corresponding idempotent summand of the $p$-local $N$-equivariant sphere, and define an $N$-Mackey functor 
    \[ \underline M_{(H),V}:=\underline{\pi}_{V-\dim V^H} e_{(H)}^N S_{N,(p)}.\]

    \begin{lemma}
        For $L\subset N$, let 
        \[
        \mathcal{C}_H (L)\cong \{Q\subseteq L:Q\sim_N H\}/L
        \]
        be the set of $L$-conjugacy classes of subgroups of $L$ that are conjugate to $H$ in $N$. Then
        \[\underline M_{H,V}(N/L)=\begin{cases}\bigoplus_{(Q)\in \mathcal{C}_H(L)}\mathbb Z_{(p)} & H \text{ is subconjugate to } L\text{ and }\epsilon_{V^Q,W_L Q}\equiv 1,\\0 &\text{otherwise}.\end{cases}\]
    \end{lemma}
    \begin{proof}
        Note that $e_{(H)}^N S_{N,(p)}$ is the idemoptent summand of $S_{N,(p)}$ such that 
        \[
        \Phi^Q e_{(H)}^N S_{N,(p)}\simeq \begin{cases}
            S_{(p)} & Q\sim^N H,\\
            \ast & \text{otherwise.}
        \end{cases}
        \]
        Thus $\operatorname{CoInd}_{N_N H}^N (\operatorname{Inf}_{W_N H}^{N_N H} F(EW_N H_+, S_{(p)})[\mathcal{F}_{H\not\subset}^{-1}])$ provides a model for $e_{(H)}^N S_{N,(p)}$. For any $L\subset N$, by \Cref{lem: double coset formula}, we have 
        \[
        \underline{M}_{H,V}(N/L)\cong \bigoplus_{(Q)\in \mathcal{C}_H(L)}\pi_{V^Q-\dim V^Q}^{W_L Q} F(EW_L Q_+, S_{(p)}).
        \] 
        It follows from \Cref{lem: stable splitting of Thom speces} that
        \[
        \pi_{V^Q-\dim V^Q}^{W_L Q} F(EW_L Q_+, S_{(p)})\cong \begin{cases}
            \pi_0 S_{(p)}\cong \mathbb Z_{(p)} & \epsilon_{V^Q,W_L Q}\equiv 1,\\
            0 & \text{otherwise.}
        \end{cases}
        \]
        This leads to the stated formula for $\underline M_{H,V}(N/L)$, since $\epsilon_{V^Q,W_L Q}\equiv 1$ if and only if $\epsilon_{V^H,W_L H}\equiv 1$.

        For any $L_1\subset L_2 \subset N$, the restriction of the conjugacy class $(Q)\in \mathcal{C}_H(L_2)$ to $L_1$ can be decomposed into a disjoint union of conjugacy classes
        \[
        \operatorname{Res}_{L_1}^{L_2}(Q)_{L_2}=\bigsqcup_{[a]\in L_1\backslash\{a\in L_2:{}^a Q\subset L_1\}/N_{L_2}Q }({}^a Q)_{L_1}.
        \]
        The restriction map on the $Q$ summand of $\underline M_{H,V}(N/L_2)$ 
        \[\pi_{V^Q-\dim V^Q}^{W_{L_2} Q} F(EW_{L_2} Q_+, S_{(p)})\to \bigoplus_{[a]}\pi_{V^Q-\dim V^Q}^{W_{{}^{a^{-1}} L_1} Q} F(EW_{{}^{a^{-1}}L_1} Q_+, S_{(p)}),\]
        is identified with the diagonal map $\mathbb Z_{(p)}\to \bigoplus_{[a]}\mathbb Z_{(p)}$.
        Stably, there is a stable transfer map between classifying spaces, which induces a transfer map between the corresponding Thom spaces. Thus we have stable maps
        \[
        EW_{L_1} H_+\wedge_{W_{L_1} H} S^{V^H}\overset{\operatorname{Res}_{W_{L_1} H}^{W_{L_2} H}}{\longrightarrow} EW_{L_2} H_+\wedge_{W_{L_2} H} S^{V^H}\overset{\operatorname{Tr}_{W_{L_1} H}^{W_{L_2} H}}{\longrightarrow} EW_{L_1} H_+\wedge_{W_{L_1} H} S^{V^H}
        \]
        such that $\operatorname{Tr}_{W_{L_1} H}^{W_{L_2} H}\circ \operatorname{Res}_{W_{L_1} H}^{W_{L_2} H}=|W_{L_2} H:W_{L_1} H|\cdot \operatorname{id}$ homologically. This induces the transfer map on every $(Q)$ summand of $\underline M_{H,V}(N/L_1)$
        \[
        \operatorname{Tr}_{L_1}^{L_2}
                =
            \begin{cases}
            |W_{L_2}Q:W_{L_1}Q|,&
            \epsilon_{V^Q,W_{L_2}Q}\equiv 1,\\
            0,&
            \epsilon_{V^Q,W_{L_2}Q}\not\equiv 1.
            \end{cases}
        \]
        Consequently, after collecting all $L_1$-conjugacy classes $(Q_i)$ contained in the $L_2$-conjugacy class (Q), the transfer to the $(Q)$-summand is
        \[\bigoplus_{Q_i} \mathbb{Z}_{(p)}\to \mathbb{Z}_{(p)},\quad (x_i)\mapsto \sum_i |W_{L_2}Q_i:W_{L_1}Q_i|x_i.\]
    \end{proof}

\begin{corollary}\label{cor:ROG graded KU_G/p local sphere}
    For any $V\in RO(G)$, 
    \[\underline{\pi}_VL_{KU_G/p}S_G\cong \bigoplus_{(H)\in Cyc/G, p\nmid |H|}\underline{\pi}_{V^H} L_{KU_{N_p}/p}S_{N_p}\otimes_{\mathbb Z_{(p)}}\underline{M}_{(H),V}.\] 
\end{corollary}
\begin{proof}
     For any cyclic subgroup $H\subset G$ such that $p\nmid |H|$, $H\subset N$ and $W_G H\cong N_p\times W_N H$. It follows from \Cref{lem:morava k local,lem: double coset formula} that for any $T\subset G$ and $V\in RO(T)$,
        \[
        \begin{aligned}
            \pi_V^T L_{K(H,1)}S_G & \cong [S^V, \operatorname{Res}_T^G L_{K(H,1)}S_G]^T\\
            &\cong \bigoplus_{[g]\in T\backslash G/N_G H} [S^V, \operatorname{CoInd}_{T\cap {}^g N_G H}^T \operatorname{Res}_{T\cap {}^g N_G H}^{{}^g N_G H} L_{K({}^gH,1)}S_{{}^gN_G H}]^T.
        \end{aligned}
        \]
        Here, in the final term, $K({}^gH,1)$ denotes the $N_G({}^gH)={}^gN_GH$-equivariant Morava K-theory spectrum
        \[ N_G({}^gH)/{}^gH_+\wedge K(1)[\mathcal{F}_{{}^gH\not\subset}^{-1}]. \]
        Consequently,
        $\operatorname{Res}_{T\cap {}^gN_GH}^{{}^gN_GH}K({}^gH,1)$ is nontrivial only if ${}^gH$ is subconjugate to $T\cap{}^gN_GH$. Since ${}^gH$ is normal in $N_G{}^gH$, this condition is equivalent to ${}^gH\subset T$. Let $T\cong P\times L \subset N_p\times N$ be the decomposition of $T$ into its Sylow $p$-subgroup and its $p$-complement.
        It follows by \Cref{prop:pushout type} that        
        \[
        \begin{aligned}
        \pi_V^T L_{K(H,1)}S_G & \bigoplus_{(Q)\in \mathcal{C}_H(L)} [S^V, \operatorname{CoInd}_{N_T Q}^T L_{K({}^gH,1)}S_{N_T Q}]^T\\
        &\cong \bigoplus_{(Q)\in \mathcal{C}_H(L)} \pi^{W_T Q}_{V^H} F(EW_T Q_+, L_{K(1)}S)\\
            & \cong \bigoplus_{(Q)\in \mathcal{C}_H(L)} [EW_T Q_+\wedge_{W_T Q} S^{V^Q}, L_{K(1)}S].
        \end{aligned}
        \]
        By \Cref{lem: stable splitting of Thom speces}, if $\epsilon_{V^H,W_L H}\equiv 1$, then
        \[
        \begin{aligned}
            \pi_V^T L_{K(H,1)}S_G &\cong 
            \bigoplus_{(Q)\in \mathcal{C}_H(L)}[(EP)_+\wedge_{P} S^{\operatorname{Res}_{P}^T V^H}, L_{K(1)}S] \\
            & \cong \bigoplus_{(Q)\in \mathcal{C}_H(L)}\pi_{V^H}^{P} L_{KU_{P}/p}S_{P}.
        \end{aligned}
        \]
        Here the final isomorphism follows from the equivalence $$L_{KU_{P}/p}S_{P}\simeq F((EP)_+, L_{K(1)}S)$$ which holds for every $p$-group $P$. If $\epsilon_{V^H,W_L H}\not\equiv 1$, then $\pi_V^T L_{K(H,1)}S_G\cong 0$.

        For any $T\subset G$, the restriction and transfer maps in $\underline{\pi}_V L_{K(H,1)}S_G$ are given by
        \[
        \begin{aligned}
            \operatorname{Res}_T^G(H): &\pi_V  L_{K(H,1)}S_G \to \pi_{\operatorname{Res}_T^G V}L_{\operatorname{Res}_T^G K(H,1)} S_T,\\
             \operatorname{Tr}_T^G(H): &\pi_{\operatorname{Res}_T^G V}L_{\operatorname{Res}_T^G K(H,1)} S_T\to \pi_V  L_{K(H,1)}S_G.
        \end{aligned}
        \]
        By \Cref{lem:tensor product}, it suffices to compute the restriction and transfer maps separately for $P\subset N_p$ and $L\subset N$.         For $P_1\subset P_2\subset N_p$, the restriction and transfer maps in $\underline{\pi}_V L_{K(H,1)}S_G$ are exactly those in $\underline{\pi}_{V^H} L_{KU_{N_p}/p}S_{N_p}$. 
        For $L \subset N$, if $H$ is not subconjugate to $L$, $\operatorname{Res}_L^N K(H,1)$ is a trivial $L$-spectrum, which implies that  
        \[
        \operatorname{Res}_L^N(H):\pi_V^N  L_{K(H,1)}S_N\to 0, \quad \operatorname{Tr}_L^N(H):0\to \pi_V  L_{K(H,1)}S_G
        \]
        are the natural projection and inclusion, respectively.
        If $H$ is subconjugate to $L$ and $\epsilon_{V^H,W_NH}\equiv 1$, then $\langle \operatorname{Res}_L^N K(H,1)\rangle = \langle K(H,1) \rangle$. 
        Therefore, the restriction map $\operatorname{Res}_L^N(H)$ is determined by the restriction of $\mathcal{C}_H(N)$ to $\mathcal{C}_L(N)$, while the transfer map $\operatorname{Tr}_L^N(H)$ is given by multiplication by $|W_N H:W_L H|$. These maps are compatible with the restriction and transfer maps in $\underline{M}_{(H),V}$, which implies that $\underline{\pi}_V L_{K(H,1)}S_G\cong \underline{\pi}_{V^H} L_{KU_{N_p}/p}S_{N_p}\otimes_{\mathbb Z_{(p)}}\underline{M}_{(H),V}$. 

        Since $L_{KU_G/p}S_G \simeq \bigvee_{\substack{(H)\in \mathrm{Cyc}/G\\ p\nmid |H|}} L_{K(H,1)}S_G$, 
        \[\underline{\pi}_VL_{KU_G/p}S_G\cong \bigoplus_{(H)\in Cyc/G, p\nmid |H|}\underline{\pi}_{V} L_{K(H,1)}S_G,\]
        which gives the stated isomorphism.
\end{proof}

    In particular, when $V=\mathbb R^n$ is the trivial $G$-representation, then $V^H=n$ is constant for all $H\in \mathrm{Cyc}$, and the computation above shows that
    \[
    \underline{\pi}_V L_{KU_G/p}S_G \cong \underline{A/J}_N \otimes \underline{\pi}_n L_{KU_{N_p}/p}S_{N_p}.
    \]
    On the other hand, we expect that \Cref{thm:KUG local spectrum second case} provides a useful framework for studying localized sphere spectra for more general finite groups. It would also be interesting to investigate analogous constructions for compact abelian Lie groups. 

\appendix
\section{The $C_4$-Mackey functor $\underline{\pi}_1 L_{KU_{C_4}/2}S_{C_4}$}\label{sec:C4}
Let $G=C_4=\langle \gamma\rangle$, $C_2=\langle \gamma^2\rangle$,
and set
\[
J:=L_{KU_{C_4}/2}S_{C_4}.
\]
For $H\leq C_4$, we write
\[
M:=\underline{\pi}_1J, \qquad M_H:=M(G/H)=\pi_1^H J.
\]
We can compute $M$ via the fiber sequence
\[
J\longrightarrow (KO_{C_4})_2^{\wedge} \xrightarrow{\psi^5-1} (KO_{C_4})_2^{\wedge}.
\]
For every subgroup $H\leq C_4$, the associated long exact sequence gives a short exact sequence
\[
\begin{aligned}
   0\longrightarrow
    \operatorname{coker}&\left(\psi^5-1:\pi_2^H (KO_{C_4})_2^{\wedge} \to \pi_2^H (KO_{C_4})_2^{\wedge}     \right)  \longrightarrow M_H \\
    &\longrightarrow \ker\left(\psi^5-1:\pi_1^H (KO_{C_4})_2^{\wedge} \to \pi_1^H (KO_{C_4})_2^{\wedge} \right) \longrightarrow 0.
\end{aligned}
\]

Let $\epsilon$ denote the non-trivial real one-dimensional representation of $C_2$, and let
\[
\sigma:C_4\longrightarrow \{\pm 1\}, \qquad \sigma(\gamma)=-1,
\]
be the sign representation of $C_4$. Thus $\sigma|_{C_2}=1$. Let $L$ denote the faithful complex one-dimensional representation of $C_4$, that is, $L(\gamma)=i$. Its underlying real representation is the faithful two-dimensional rotation representation, which we denote by $\lambda$.

\Cref{lem: group structure} implies that this short exact sequence is pointwise split as a sequence of abelian groups. Therefore, the values of $M$ are
\[
M_e\cong \mathbb F_2\{a,b\}, \quad M_{C_2}\cong \mathbb F_2\{a_1,a_\epsilon,b_1,b_\epsilon\},
\]
and
\[
M_{C_4}\cong \mathbb F_2\{A_1,A_\sigma,B_1,B_\sigma\} \oplus \mathbb Z/4\{c\}.
\]
Here $a, a_1, A_1$ are the Hurewicz image of the Hopf element $\eta$, and 
\[
a_\epsilon=h_{C_2}(\operatorname{Tr}_e^{C_2}(1)\eta-\eta),\quad A_\sigma=h_{C_4}(\operatorname{Tr}_{C_2}^{C_4}(1)\eta-\eta).
\]
The classes $b$, $b_1$, $b_\epsilon$, $B_1$, and $B_\sigma$ come from the classes associated to $\eta^2$ in $\mathrm{coker}(\psi^5-1)$. The class
\[
c\in M_{C_4}
\]
is the class of the cokernel associated to $\beta L$; more explicitly, if
\[
r_{\mathbb R}:KU\longrightarrow KO
\]
denotes realification, then $c$ is represented by $r_{\mathbb R}(\beta L)$, where $\beta\in \pi_2KU$ is the complex Bott class.

The restriction maps are given by 
\[
\operatorname{Res}_e^{C_2}(a_1)=a, \qquad \operatorname{Res}_e^{C_2}(a_\epsilon)=a,
\]
\[
\operatorname{Res}_e^{C_2}(b_1)=b, \qquad \operatorname{Res}_e^{C_2}(b_\epsilon)=b,
\]
and
\[
\operatorname{Res}_{C_2}^{C_4}(A_1)=a_1, \qquad \operatorname{Res}_{C_2}^{C_4}(A_\sigma)=a_1,
\]
\[
\operatorname{Res}_{C_2}^{C_4}(B_1)=b_1, \qquad \operatorname{Res}_{C_2}^{C_4}(B_\sigma)=b_1,
\]
\[
\operatorname{Res}_{C_2}^{C_4}(c)=b_\epsilon.
\]
The transfer maps are given by 
\[
\operatorname{Tr}_e^{C_2}(a)=a_1+a_\epsilon, \qquad \operatorname{Tr}_e^{C_2}(b)=b_1+b_\epsilon,
\]
\[
\operatorname{Tr}_{C_2}^{C_4}(a_1)=A_1+A_\sigma, \quad \operatorname{Tr}_{C_2}^{C_4}(b_1)=B_1+B_\sigma, \qquad \operatorname{Tr}_{C_2}^{C_4}(b_\epsilon)=0.
\]
All these results are determined by the restriction and transfer maps in $RO(C_4)$, and it remains to determine $\operatorname{Tr}_{C_2}^{C_4}(a_\epsilon)$. Let
\[
J_e:=\operatorname{hofib}\left(
KO_2^\wedge\xrightarrow{\psi^5-1}KO_2^\wedge
\right),
\qquad
J_{e,\mathbb C}:=\operatorname{hofib}\left(
KU_2^\wedge\xrightarrow{\psi^5-1}KU_2^\wedge
\right).
\]
There are splittings
\[
\bigl((KO_{C_2})_2^\wedge\bigr)^{C_2}
\simeq
(KO_2^\wedge)^{(1)}\vee (KO_2^\wedge)^{(\epsilon)}
\]
and
\[
\bigl((KO_{C_4})_2^\wedge\bigr)^{C_4}
\simeq
(KO_2^\wedge)^{(1)}
\vee
(KO_2^\wedge)^{(\sigma)}
\vee
(KU_2^\wedge)^{(\lambda)},
\]
where the superscripts indicate the irreducible representations indexing
the corresponding summands. Since the Adams operation \(\psi^5\)
preserves each of these summands, these splittings induce equivalences
\[
(J_{C_2})^{C_2}
\simeq
J_e^{(1)}\vee J_e^{(\epsilon)}, \qquad 
(J_{C_4})^{C_4}
\simeq
J_e^{(1)}\vee J_e^{(\sigma)}\vee J_{e,\mathbb C}^{(\lambda)}.
\]
The induction homomorphism on real representation rings satisfies
\[
\operatorname{Tr}_{C_2}^{C_4}(1)=1+\sigma,
\qquad
\operatorname{Tr}_{C_2}^{C_4}(\epsilon)=\lambda.
\]
Consequently, under the above splittings, the transfer map on fixed
points is given by
\[
\operatorname{Tr}_{C_2}^{C_4}
\bigl(x\otimes 1,y\otimes\epsilon\bigr)
=
\bigl(
x\otimes 1,\,
x\otimes\sigma,\,
\operatorname{c}(y)\otimes\lambda
\bigr),
\]
where
\[
\operatorname{c}\colon KO_2^\wedge\longrightarrow KU_2^\wedge
\]
is the complexification map. Since complexification commutes with
\(\psi^5\), it induces a map
\[
\operatorname{c}_j\colon J_e\longrightarrow J_{e,\mathbb C}.
\]
It follows that the restriction of
\(\operatorname{Tr}_{C_2}^{C_4}\) to the \(J_e^{(\epsilon)}\)-summand is
identified with \(\operatorname{c}_j\), with image in the
\(J_{e,\mathbb C}^{(\lambda)}\)-summand. Since \(a_\epsilon\) has order
\(2\), while $\pi_1J_{e,\mathbb C}^{(\lambda)} \cong \mathbb Z/4\{c\}$,
the only possible values of $\operatorname{Tr}_{C_2}^{C_4}(a_\epsilon)$ are $0$ and $2c$.

\begin{proposition}
    $\operatorname{Tr}_{C_2}^{C_4}(a_\epsilon)=2c$.
\end{proposition}
\begin{proof}
    Note that $a_\epsilon$ is a lifting of $\eta\epsilon\in \pi_1^{C_4}KO_{C_4}$ and as the construction in the proof of \Cref{lem: group structure}, $a_\epsilon$ is the Hurewicz image of $\operatorname{Tr}_e^{C_2}(\eta)-\eta$ for the Hopf element $\eta\in \pi_1 S$. 
    Let
    \[
    J_{\mathbb C} := \operatorname{hofib} \left( (KU_{C_4})^\wedge_2\xrightarrow{\psi^5-1}(KU_{C_4})^\wedge_2 \right), 
    \]
    the complexification map $KO_{C_4}\to KU_{C_4}$ induces a $C_4$-map $f:J\to J_{\mathbb{C}}$. Here 
    \[
    \pi_1^{C_2} J_{\mathbb{C}}\cong \mathbb{Z}/4\{\beta, \epsilon\beta\}, \qquad \pi_1^{C_4} J_{\mathbb{C}}\cong \mathbb{Z}/4\{L^i\beta:0\leq i\leq 3\}.
    \]
    The homomorphism $f_*: \underline{\pi}_1 J \to \underline{\pi}_1 J_{\mathbb{C}}$ satisfies that 
    \[
    f_*(a_\epsilon)= (\operatorname{Tr}_e^{C_2}(1)-1) f_*(\eta) = \epsilon f_*(\eta) =\epsilon \cdot 2\beta, 
    \]
    and 
    \[ 
    f_*(\operatorname{Tr}_{C_2}^{C_4}(a_\epsilon))=\operatorname{Tr}_{C_2}^{C_4} 2\epsilon \beta = 2 \beta (L+L^3) \neq 0.
    \]
    Therefore $\operatorname{Tr}_{C_2}^{C_4}(a_\epsilon) = 2c \neq 0$.
\end{proof}

Finally, we record the nonequivariant $f_*(\eta)=2\beta$ used above. Restrict to the trivial group $\{e\}$,
\[
\pi_1J_{e,\mathbb C} \cong \mathbb Z/4\{\beta\}.
\]
Since $\eta\in \pi_1 L_{KU/2}S$ is the Hurewicz image of the Hopf element $\eta\in \pi_1S$, $f_*\eta$ is the element represented by 
\[
 S^1\xrightarrow{\eta} S^0 \to J_e \xrightarrow{f} J_{e,\mathbb{C}}.
\]
Consider the cofiber sequence $S^1\xrightarrow{\eta} S^0 \to C\eta\simeq \Sigma^{-2}\mathbb CP^2$. 
Let $u=[\mathcal O(1)]-1$ be the generator of $\widetilde{KU}^*(\mathbb{C}P^2)$ such that $\widetilde{KU}(\mathbb{C}P^2)\cong \mathbb{Z}[u]/u^3$, we have $\psi^5(u)=5u+10u^2$.
Stably $f_*\eta$ is the obstruction of the extension of $S^0\to J_e$ via $S^0\to C\eta$. Consider the diagram 
\[
\begin{tikzcd}
    S^1 \arrow[r,"\eta"] & S^0 \arrow[r]\arrow[d]  & C\eta\arrow[ldd,"\beta^{-1}u"]\\
    & J_{e,\mathbb{C}} \arrow[d] &\\
    & KU_2^\wedge &
\end{tikzcd},
\]
$\beta^{-1}u$ gives an extension of $S^0\to KU_2^\wedge$. Since $(\psi^5-1)(\beta^{-1}u)=2\beta^{-1}u^2$, $\beta^{-1}u$ cannot lifts to $J_{e,\mathbb{C}}$, which implies that $f_*\eta\neq 0\in \pi_1 J_{e,\mathbb{C}}$, thus $f_*\eta=2\beta$.

\bibliographystyle{alpha}
\bibliography{homotopy.bib}

\end{document}